\documentclass[a4paper,english]
{article}

\usepackage{amssymb,amsmath}
\makeindex

\input xy  \xyoption{all}

\newtheorem{deff}{Definition}[section]
\newtheorem{lemma}[deff]{Lemma}
\newtheorem{theorem}[deff]{Theorem}
\newtheorem{corollary}[deff]{Corollary}

\newtheorem{proposition}[deff]{Proposition}

\newtheorem{facts}[deff]{Facts}
\newtheorem{em-example}[deff]{Example}
\newtheorem{em-def}[deff]{Definition}        
\newtheorem{em-remark}[deff]{Remark}         
\newtheorem{em-question}[deff]{Question}

\newenvironment{example}{\begin{em-example} \em }{ \end{em-example}}
\newenvironment{definition}{\begin{em-def} \em  }{ \end{em-def}}
\newenvironment{remark}{\begin{em-remark} \em }{\end{em-remark}}
\newenvironment{question}{\begin{em-question}\em }{\end{em-question}}
\newenvironment{proof}{\noindent {\it Proof}.}{\QED \smallskip}

\newcommand{\ol}[1]{\overline{#1}}

\newcommand{\wt}[1]{\widetilde{#1}}

\newcommand\QED{\hfill QED \medskip}

\newcommand{\tcal}{\cal {T}}

\def\hull#1{\langle#1\rangle}

\def\ker{\mathop{\rm ker}}

\def\:{\nobreak \hskip .1111em\mathpunct {}\nonscript \mkern
-\thinmuskip {:}\hskip .3333emplus.0555em\relax}

\def\nbd{neighborhood}
\catcode`\@=12
\def\T{{\mathbb T}}
\def\Q{{\mathbb Q}}

\def\Z{{\mathbb Z}}
\def\N{{\mathbb N}}
\def\R{{\mathbb R}}
\def\Q{{\mathbb Q}}
\def\P{{\mathbb P}}

\def\Prm{\P}

\def\cont{\mathfrak c}

\title{{ Locally minimal topological groups 2}
\thanks{The first named author was partially supported by MTM 2008-04599. The other authors were partially supported by Spanish MICINN MTM2009-14409-C02-01. The third author  was partially supported also by SRA, grants P1-0292-0101 and J1-9643-0101. }
\\{ \normalsize To the memory of Ivan Prodanov (1935-1985)}
}
\author{
Lydia Au\ss{}enhofer \\{\small\em Welfengarten 1,} {\small\em 30167 Hannover, Germany.} {\small\em e-mail: aussenho@math.uni-hannover.de}\\
 M. J. Chasco \\ {\small\em Dept. de F\'{\i}sica y Matem\'{a}tica Aplicada,} {\small\em Universidad  de Navarra, Spain.}
{\small\em e-mail: mjchasco@unav.es}\\
 Dikran Dikranjan \\{\small\em Dipartimento di Matematica e Informatica, Universit\`{a} di Udine, Italy.} {\small\em e-mail: dikranja@dimi.uniud.it}\\
 Xabier Dom\'{\i}nguez\\ {\small\em Departamento de M\'{e}todos Matem\'{a}ticos y de
Representaci\'{o}n,} {\small\em Universidad de A Coru\~{n}a, Spain.} {\small\em e-mail: xdominguez@udc.es}
}

\begin{document}

\maketitle

\begin{abstract}
We continue in this paper the study of locally minimal groups started in \cite{LocMin}.

The minimality criterion for dense subgroups of compact
groups is extended to local minimality. Using this criterion we characterize the compact
abelian groups containing dense countable locally minimal subgroups,
as well as those  containing dense locally minimal subgroups of countable free-rank.
We  also characterize the compact abelian groups whose torsion part is dense
 and locally minimal.

We call a topological group $G$ {\it almost minimal} if it has a closed, minimal normal subgroup $N$ such that the quotient group $G/N$ is uniformly free from small subgroups. The class of almost minimal groups includes all locally compact groups, and is contained in the class of locally minimal groups. On the other hand, we provide examples of countable precompact metrizable locally minimal groups which are not almost minimal. Some other significant properties of this new class are obtained.

Keywords: locally minimal group, minimal group,  locally essential subgroup, essential subgroup,   almost minimal group,  GTG set, locally GTG group,  group without small subgroups,  locally quasi-convex group

MSC 22A05, 22B05
\end{abstract}
\section{Introduction}

Minimal topological groups were introduced  independently by Choquet, Do\" \i tchinov \cite{Do1} and
Stephenson \cite{St}: a Hausdorff topological group $(G,\tau)$ is called minimal
if there exists no Hausdorff group topology on $G$ which is strictly coarser than $\tau$.
Various generalizations of minimality have been defined and intensively studied (\cite{meg}, \cite{DMeg}).
The  notion of local minimality
was introduced by Morris and  Pestov in \cite{MP} (see also \cite{TaB} and \cite{DM}).
A Hausdorff topological group is locally minimal with respect to
a neighborhood of zero $U$ if it does not admit a strictly coarser Hausdorff group topology for which $U$ is still a
neighborhood of zero.
This paper continues the study started in \cite{LocMin}, where the authors went further into some aspects of local minimality. We focus here mainly on two aspects: finding a ``local minimality criterion"{} and introducing and studying the class of locally minimal groups that can be obtained as extensions of  minimal groups via UFSS quotient groups.

The problem of finding a local minimality criterion comes out as a natural question, if we take into account the crucial role which was played in the theory of minimal groups by the so-called ``minimality criterion"{}, due to
Banaschewski  \cite{B}, Stephenson \cite{St} and Prodanov \cite{P1}; namely,  the characterization of those
dense subgroups of a minimal group that are minimal as topological
subgroups (we recall it in Theorem \ref{Min_Crit}). In \S \ref{loc_ess_loc_min} we obtain a counterpart of this
criterion (Theorem \ref{crit}) for local minimality based on the new notion of a
{\em locally essential subgroup} (see Definition \ref{LocEsse}).

The rest of Section 3 and the whole of Section 4 are devoted to characterizing different classes of precompact locally minimal abelian groups using the local minimality criterion. Note that the study of local minimality in locally precompact abelian groups can be reduced to the corresponding problem for precompact groups; this is a consequence of the local minimality criterion and the fact that locally precompact abelian groups are exactly the subgroups of locally compact abelian groups, whose structure is well known. In Subsection 3.2 we characterize the dense locally minimal subgroups of finite-dimensional compact groups (Proposition  \ref{corol0}). In Section 4 we characterize the  compact abelian groups which contain a ``small" dense locally minimal subgroup, where by ``small"{} we mean countable (Theorem \ref{theocorol1}), of countable free rank (Theorem \ref{corol2}) or torsion (Theorem \ref{ExTor}). Each one of these results is based on a known counterpart for minimal groups. Actually, we show that in many cases precompact abelian locally minimal groups turn out to be minimal (such is the case for subgroups of torsion-free compact groups), or close to minimal.

%

We next consider locally quasi-convex groups. We recall the definition and basic properties of this class in Subsection 5.1. A natural example of a locally quasi-convex group is the underlying group of a locally convex space; locally compact abelian and precompact abelian groups are always locally quasi-convex.
The first link between the properties of local quasi-convexity and local minimality that we explore here concerns the class of UFSS groups. A Hausdorff topological group is UFSS (Uniformly Free from Small Subgroups) if its topology is generated by a single neighborhood of zero in a natural analogous way as the unit ball of a normed space determines its topology.  As we showed in  \cite{LocMin},  these UFSS groups constitute an important subclass of locally minimal groups. Continuing the study of the algebraic structure of locally minimal groups that we started in \cite[\S 5.2]{LocMin}, we show that for every increasing sequence $(m_n)$ of natural numbers and every prime $p$, the group $\bigoplus_{n\in \N} \Z(p^{m_n})$ admits a non-discrete locally quasi-convex UFSS group topology (Theorem \ref{Ex_Lydia}). Note that by contrast, the group $\Z(p^\infty)$ does not admit a minimal group topology for any prime $p$ ((3.5.4) in \cite{DPS}).

The second main goal of this paper is to replace local minimality by a stronger property that still covers local compactness, UFSS and minimality (in the abelian case), but goes closer to them in the following natural sense: A topological group $G$ is called {\em almost minimal} if it has a closed, minimal normal subgroup $N$ such that the quotient group $G/N$ is UFSS. Every almost minimal abelian group is locally minimal (Theorem \ref{char_alm_min}). Further, it will be shown that  every locally quasi-convex locally minimal
group can be embedded into an almost minimal group (Theorem \ref{embedding}). This embedding allows us to show that
every complete locally quasi-convex locally minimal group is already almost minimal
 (Corollary \ref{completion_slm}).  In particular, every locally compact abelian group is almost  minimal.
Complete almost minimal abelian groups 
are \v Cech complete, hence $k$-spaces and also Baire spaces.

 We give an example of a countable precompact locally minimal group which is not almost minimal (Example \ref{counterexample}), and show that one can actually construct ${\mathfrak c}$-many pairwise non-isomorphic groups with these properties (Proposition \ref{Rem_Anti_StrongLM_}). These examples show that 
 locally quasi-convex locally minimal groups which are not complete need not be almost minimal, and also that a locally essential dense subgroup of a compact group may fail to be almost minimal.
\section{Background }

\paragraph{Notation and terminology}

The subgroup generated by a subset $X$ of a group $G$ is denoted by $\hull{X}$, and $\hull{x}$ is the cyclic subgroup of $G$ generated
by an element $x\in G$. The abbreviation $K\leq G$ is used to denote a subgroup  $K$ of $G$.
Since we deal mainly with abelian groups, we use additive notation, and denote by $0$ its neutral element. For a subset $A$ of an abelian group $G$ and $m\in \N$, we write $A+\buildrel m \over \dots+A:=\{ a_1+\cdots+a_m\, : \, a_i \in A,\, i\in \{1,\cdots, m\}\}.$

We denote by $\N$ and $\Prm$ the sets of positive natural numbers
and primes, respectively; by $\N_0$ the set $\N\cup\{0\}$; by $\Z$ the integers, by $\Q$ the rationals, by $\R$ the reals, and by $\T$ the unit circle group
which is identified with $\R/\Z$.
We will denote by $q$ the canonical projection from $\R$ to $\T.$ The cyclic group of order $n>1$ is denoted by $\Z(n)$. For a  prime $p$ the symbol $\Z(p^\infty)$
stands for the  quasicyclic $p$-group and $\Z_p$ stands for the $p$-adic integers. The cardinality of the continuum $2^{\omega}$ will be also denoted by $\cont$.

The {\it torsion part\/} $t(G)$ of an abelian group $G$ is the set $\{x\in G: nx=0\ {\rm for\ some}\ n\in\N\}$. Clearly, $t(G)$ is a
subgroup of $G$.  For any $p\in {\mathbb P}$, the {\em $p$-primary component} $G_p$ of $G$ is the subgroup of $G$ that consists of all $x\in G$
satisfying $p^nx=0$ for some positive integer $n$. For every $n\in\N$, we put $G[n]=\{x\in G: nx=0\}$. We say that $G$ is {\it bounded}
if $G[n]=G$ for some $n\in \N$. If $p\in {\mathbb P},$ the $p$-{\it rank\/} of $G$, $r_p(G)$, is defined as the cardinality of
a maximal independent subset of $G[p]$ (see \cite[Section 4.2]{Rob}). The group $G$ is {\it divisible} if $nG=G$ for every
$n\in \N$, and {\it reduced}, if it has no divisible subgroups beyond $\{0\}$.  The {\em free rank\/} $r_0(G)$ of the group $G$ is the
cardinality of a maximal independent subset of $G$. The {\em socle} of $G,$ $Soc(G),$ is the subgroup of $G$ generated by all elements of prime order, i. e. $Soc(G)=\bigoplus_{p\in {\mathbb P}}\,G[p].$

We denote by ${\cal V}_\tau(0)$ (or simply by ${\cal V}(0)$) the filter of neighborhoods
of the neutral element $0$ in a topological group $(G, \tau)$. Neighborhoods are not necessarily open.

For a topological group $G$ we denote by $\widetilde{G}$ the Ra\u \i kov completion of $G$. We recall here that a group $G$ is {\it
precompact\/} if $\widetilde{G}$ is compact (some authors prefer the term ``totally bounded").

We say that a topological group $G$ {\em is linear} or {\em is linearly topologized} if it has a neighborhood basis at $0$ formed by open subgroups.


By a  {\em character} on an abelian topological group $G$ it is commonly understood a continuous homomorphism from $G$ into $\mathbb{T}$. Under pointwise addition the  characters on $G$ constitute a group  $G^{\wedge}$ called the {\em dual group} or
{\em character group} of $G$.
%
%
%
%
Endowed with the compact open topology $ \tau_{co}$, it becomes a Hausdorff topological group. A basis of
neighborhoods of the neutral element for the  compact open topology $\tau_{co}$ is given by the sets $(K, \mathbb{T}_+): = \{\chi \in
G^{\wedge}: \chi(K) \subseteq \mathbb{T}_+\}$, where $\mathbb{T}_+ :=q([1/4,1/4])$
and $K$ is a compact subset of $G$. The natural evaluation mapping from $G$ to $(G^{\wedge},\tau_{co})^{\wedge}$ will be denoted by $\alpha_G$, that is, $\alpha_G(x)(\chi)=\chi(x)$ for every $\chi \in G^{\wedge}.$ Obviously, $\alpha_G$ is injective if and only if the characters of $G$ separate the points; in that case we say that $G$ is {\em maximally almost periodic.}
\par

Let $U$ be a symmetric subset of a group $(G,+)$ such that $0\in U,$ and $n\in \N$. We define
$(1/n)U: = \{x\in G\,:\, kx\in U \;\forall k\in\{1,2,\cdots, n\}\}$ and $U_\infty:=\{x \in G\,:\, nx \in U \;\forall n\in \N\}.$ Note that for a symmetric subset $U$ of a vector space we have  $\displaystyle(1/n)U=\bigcap_{k=1}^n \frac{1}{k}U.$

All unexplained terms concerning topological spaces and topological groups can be found in \cite{Eng} and \cite{Arh}. For background on abelian groups, see \cite{Fuc} and \cite{Rob}.
\section{Local minimality criterion} \label{loc_ess_loc_min}
\subsection{Locally essential subgroups and  local minimality criterion}
A Hausdorff
topological group $(G, \tau)$ is {\em locally minimal} if there exists a
neighborhood $V$ of $0$ such that whenever $\sigma\leq \tau$  is a Hausdorff
group topology on $G$ such that $V$ is a $\sigma$-neighborhood of $0$, then
$\sigma=\tau$. If we want to point out that the neighborhood $V$ witnesses local
minimality for $(G,\tau)$ in this sense, we say that $(G,\tau)$ is
{\em $V$-locally minimal.} If local minimality of a group $G$ is witnessed by some $V\in {\cal V}_\tau(0)$, then every smaller $U\in
{\cal V}_\tau(0)$ witnesses local minimality of $G$ as well. Examples for locally minimal groups are minimal groups and  locally compact groups. Every open subgroup of a locally minimal group is locally minimal (\cite[Proposition 2.4]{LocMin}).

We recall here the notion of an essential subgroup: a subgroup $H$ of a topological group $G$ is {\it essential} if
$$
 \mbox{ for every closed normal subgroup }N\mbox{ of }G, \;\;
 [H\cap N=\{0\}\;\; \Rightarrow\;\;  N=\{0\}]
$$
The following minimality criterion played a central role in the study of minimal groups (\cite[Theorem 2.5.1]{DPS}):

\begin{theorem}\label{Min_Crit} {\rm (\cite{B}, \cite{P1}, \cite{St})} Let $G$ be a  Hausdorff topological group and let $H$ be a dense subgroup of $G$.
Then $H$ is minimal iff $G$ is   minimal and $H$ is essential in $G$.
\end{theorem}

We propose now a ``local" version  of essentiality:

\begin{definition}\label{LocEsse}  Let $H$ be a subgroup of a topological group $G$. We say
that $H$ is {\em locally essential} in $G$ if there exists a neighborhood $V$ of $0$ in $G$
such that  $H\cap N=\{0\}$ implies $N=\{0\}$ for all closed normal subgroups $N$ of $G$ contained in $V$.

When necessary, we shall say $H$ {\em is locally essential with respect to} $V$ to indicate that
$V$ witnesses local essentiality.
Note that
if $V$ witnesses local essentiality,
then any smaller neighborhood of zero does, too.\end{definition}

\begin{remark}\label{rem_loc_ess} (a) Clearly, essential subgroups are also locally essential (simply take $V=G$).

(b)  The essentiality coincides in the case of discrete groups with the known notion of essentiality in algebra. In contrast with this,
every subgroup of a discrete group is locally essential, i.e., local essentiality becomes vacuous in the discrete case.

\end{remark}


\begin{proposition}\label{loc_ess}
Let $G$ be a topological abelian group and let $H$ be a subgroup of $G$.
\begin{itemize}

 \item[(a)] If $H$ is locally essential in $G$ with respect to a neighborhood $U$ of 0, then for every closed subgroup $N$ of $G$, $N\cap H$ is locally essential in $N$ with respect to $N\cap U.$
 \item[(b)] If $V$ is an  open subgroup of $G$, then   $H$ is a  locally essential subgroup of $G$ with respect to $V$ iff $H\cap V$ is an essential subgroup of $V$.

 \item[(c)] If $H$ is locally essential in $G$ with respect to a neighborhood $U$ of 0 in $G$, then for every closed and linearly topologized subgroup $N$ of $G$, there exists an open subgroup $V$ of $N$ with $V\subseteq U$ and such that $H\cap V$ is essential in $V$.
\end{itemize}
\end{proposition}

\begin{proof}  (a)  is trivial and (b) is an easy consequence of the fact that a nontrivial subgroup of $V$ is closed in $G$ if and only if it is closed in $V$.

(c) By (a) $N\cap H$ is a locally essential subgroup of $N$ with respect to the neighborhood $U\cap N$ of 0 in $N$. Now one can pick an open subgroup $V$ of $N$
contained in $U \cap N$ and apply (b) to $V$,  $N\cap H$ and $N$.
\end{proof}

Proposition \ref{loc_ess}(c), when applied to the particular case $N=G$, shows that the notion of local essentiality can be simplified further when the  group $G$ is abelian and has a linear topology. This
is also a good evidence that the term ``locally essential" has been chosen appropriately.

We now give a criterion for local minimality.

\begin{theorem}\label{crit} Let $H$ be a dense subgroup of a topological
group $G$. Then $H$ is locally minimal iff $G$ is locally minimal and $H$ is locally
essential in $G$.
\end{theorem}

\begin{proof} Assume first that $H$ is locally minimal. Let $\tau$ be the
topology of $G$ and let $W\in {\cal V}_{\tau}(0)$ be a closed neighborhood
such that $V=W\cap H$ witnesses local minimality of $H$.

\medskip

(a)  Every $W_1 \in {\cal V}_{\tau}(0)$ satisfying $W_1+W_1\subseteq W$ witnesses local essentiality of $H$ in $G$.

Let $N$ be a closed normal subgroup of $G$ contained in $W_1$ and such that $N\cap H=\{0\}$, then $H\cap (W_1+N) \subseteq V$. Denote by $\sigma$ the
group topology induced on $H$ by the restriction of the canonical map  $f: G\to G/N$. A basic neighborhood of zero for this topology is
$H\cap (U+N)$ where $U$ is a zero neighborhood in $(G,\tau)$. In particular $V$ is a neighborhood of $0$ also in $\sigma$. The topology $\sigma$ is
Hausdorff, since $N$ is closed and $H\cap N=\{0\}$.   By the local minimality of  $H$ we conclude that $\sigma =\tau|_H$. Let us see that this implies $N=\{0\}$:
  The sets of the form $\ol{(U+N)\cap H}$ where $U$ runs through the open neighborhoods of $0$ in $(G,\tau)$ form a neighborhood basis of $0$ in
$(G,\tau)$ and contain $N$, since $N+U$ is open. Since $\tau $ is Hausdorff, it follows that $N=\{0\}$. Consequently $W_1$ witnesses local essentiality of  $H$.

(b) $W$ witnesses local minimality of $G$.

Indeed, let  $\sigma \leq \tau$ be a Hausdorff group topology on $G$ such that $W\in {\cal V}_{\sigma}(0)$. Then $\sigma|_H$ is
Hausdorff, $\sigma|_H\leq \tau|_H$ and $V$ is a $\sigma|_H$-neighborhood of $0$ in $H$. Since $H$ is locally minimal
(with respect to $V$) we can claim that $\sigma |_H=\tau|_H$. For every $U\in {\cal V}_{\tau}(0)$ there exists $U_1\in {\cal
V}_{\tau}(0)$ such that $U_1+U_1\subseteq U$. Now $U_1\cap H$ is $\tau|_H$-open and hence also $\sigma |_H$-open. Then there exists a
$\sigma$-open $O\in {\cal V}_{\sigma}(0)$ such that $O\cap H \subseteq U_1$. Then the  $\tau$-density of $H$ in $G$ yields
$O\subseteq \overline{O\cap H}^\tau\subseteq U_1+U_1\subseteq U$. Hence $U\in {\cal V}_{\sigma}(0)$.

\medskip
In the opposite direction, if $V\in {\cal V}_{\tau}(0)$ simultaneously witnesses local minimality of $G$ and local essentiality of $H$ in $G$,
then for every neighborhood $V_1\in {\cal V}_{\tau}(0)$ with  $V_1+V_1\subseteq V$ the neighborhood $V_1\cap H$ witnesses local minimality of $H$.

Indeed, fix a Hausdorff group topology $\sigma$ on $H$ such that $V_1\cap H\in{\cal V}_{\sigma}(0)$ and $\sigma\le \tau|_H.$ We need to show that $\sigma=\tau|_H.$

From the fact that $\tau$ is a group topology, it is easy to derive that the family $\{\overline{U}^{\tau}:U\in {\cal V}_{\sigma}(0)\} $ is a basis of neighborhoods of $0$ for a group topology $\tau'$ on $G$.

We have $\tau'\le \tau$ since for every $U\in {\cal V}_{\sigma}(0)$ there exists a $\tau$-open $W\in {\cal V}_{\tau}(0)$ such that $W\cap H\subseteq U$ and in particular $\overline{U}^{\tau}\supseteq \overline{W\cap H}^{\tau}=\overline{W}^{\tau}\supseteq W$ (here we use that $H$ is dense and $W$ is open).

Moreover, $V\in {\cal V}_{\tau'}(0)$ since
$$V\supseteq V_1+V_1\supseteq \overline{V_1}^{\tau}\supseteq \overline{V_1\cap H}^{\tau}\in {\cal V}_{\tau'}(0)$$

Finally, $\tau'$ is a Hausdorff group topology: Note that the subgroup $$\overline{\{0\}}^{\tau'}=\bigcap_{U\in {\cal V}_{\sigma}(0)} \overline{U}^{\tau}$$
is $\tau$-closed, normal and contained in $V$ (as $V$ is a $\tau'$-neighborhood of $0$). Moreover,
$$\Big(\bigcap_{U\in {\cal V}_{\sigma}(0)} \overline{U}^{\tau}\Big)\cap H=\bigcap_{U\in {\cal V}_{\sigma}(0)} \overline{U}^{\tau}\cap H=\bigcap_{U\in {\cal V}_{\sigma}(0)} \overline{U}^{\tau|_H}\subseteq \bigcap_{U\in {\cal V}_{\sigma}(0)} \overline{U}^{\sigma}=\{0\}$$
and from local essentiality of $H$ we deduce $\bigcap_{U\in {\cal V}_{\sigma}(0)} \overline{U}^{\tau}=\{0\}.$

Since $(G,\tau)$ is $V$-locally minimal we deduce $\tau=\tau'$. This trivially implies $\sigma\ge\tau|_H.$

\end{proof}

 \begin{example}\label{t_m}

For any nonempty set $B$ of primes put $K_B=\prod_{p\in B}\Z(p)$. Fix an arbitrary infinite set $A$ of primes and consider a dense subgroup $G$ of $K_A$. Then $G$ is locally minimal iff $G$ contains all but a finite number of the subgroups $\Z(p)$, $p\in A$. [According to Theorem \ref{crit}, $G$ is locally minimal iff $G$ is locally essential in $K_A.$ Every neighborhood of $0$ in 
$K_A$ contains an open subgroup of $K_A$ of the form  $V=K_{A'},$ where $A'$ is a 
cofinite subset of $A$. By item (b) of Proposition \ref{loc_ess}, $G$ is locally essential with respect
to  $V$  iff $G \cap V$ is an essential subgroup of $V$. A simple application of the Chinese Remainder Theorem shows that every non-zero closed subgroup of  $K_{A'}$  has the form  $K_B$, where $B$ is a nonempty subset of $A'$. Hence, $G \cap V$ is an essential subgroup of $V$ iff $G$ contains all subgroups $\Z(p)$, $p\in A'$.]
\end{example}

Next we highlight some consequences of Theorem \ref{crit}:
\begin{corollary} \label{Crit2}
\begin{itemize}
  \item[(a)] The class of locally  minimal groups is closed under completion. Actually, a more general result is true: if $G$ has a dense locally minimal subgroup $H$, then any subgroup $H'$ with  $H\le H'\le G$ is locally minimal, too.
  \item[(b)] Let $G$ be a group that is either locally compact or minimal. Then a dense subgroup $H$ of $G$ is locally minimal iff $H$ is locally essential in $G$.
 \item[(c)]  If $G$ is a locally minimal topological abelian group and if $H$ is a dense  subgroup of $G$ such that for some  linearly topologized closed non-discrete subgroup $N$ of $G$ one has $H\cap N=\{0\}$, then $H$ is not locally minimal.
 \end{itemize}
\end{corollary}
\begin{proof} (a) $H$ is locally essential in $G$, so is the bigger subgroup $H'$. (b) is immediate.
(c) Assume that $H$ is locally minimal. Then $H$ must be locally essential in $G$ by Theorem \ref{crit}. This will be witnessed by some open
 neighbourhood $U$ of $0$ in $G$. By Proposition \ref{loc_ess}(c), $U$ contains an open subgroup $V$ of $N$ such that $H\cap V$ is an essential subgroup of $V$. But $H\cap V\subseteq H\cap N=\{0\}$, hence $V=\{0\}$, a contradiction.
\end{proof}

Recall that a topological group $(G,\tau)$ is called {\em NSS group} (No Small Subgroups) if a suitable neighborhood
$V\in{\cal V}(0)$ contains only the trivial subgroup.
The topological group $(G,\tau)$ is called {\em NSnS group} (No Small normal Subgroups) if a suitable neighborhood $V\in{\cal V}(0)$ contains only the trivial normal subgroup.
The distinction between NSS and NSnS will be necessary only when we consider non-abelian groups or non-compact  groups (\cite[32.1]{stroppel}). It is clear that if $G$ has no small normal subgroups,  then  every  subgroup  $H$ of $G$ is locally essential, so we obtain the following corollary.


\begin{corollary}\label{Crit3}  Let $H$ be a dense subgroup of a locally  minimal NSnS group $G$. Then  $H$   is locally minimal.
\end{corollary}

\begin{remark} Corollary \ref{Crit3} provides a direct proof of the fact that every subgroup of a Lie group is locally minimal (note that any closed subgroup of a Lie group is still a Lie group). \end{remark}

It is known that a group having a dense NSnS subgroup need not be NSnS (e.g., take a dense cyclic subgroup of $\T^\N$ \cite[Remark 2.16(a)]{LocMin}).
We now show that this cannot occur when the dense subgroup is locally essential.

\begin{proposition} \label{denselocess}
If $H$ is a dense locally essential subgroup of a group $G$, then $H$ is NSnS iff $G $ is NSnS.
\end{proposition}

\begin{proof} Any dense (not necessarily locally essential) subgroup of a NSnS group is NSnS (this is Lemma 2.12(c) in  \cite{LocMin}).
Suppose now that $H$ is NSnS. Assume that

(i) the zero neighborhood $W$ in $G$ is such that $W\cap H$ witnesses NSnS of $H$; and

(ii) the zero neighborhood $V$ in $G$ witnesses  local essentiality of $H$ in $G$.

We show that $W\cap V$ does not contain any nontrivial closed normal subgroup of $G$. Indeed, if
$N$ is a closed normal subgroup of $G$ and $N\subseteq W\cap V,$ in particular $N\cap H$
is a closed normal subgroup of $H$ contained in $W\cap H$, hence $N\cap
H=\{0\};$ since $N\subseteq V,$ by local essentiality we deduce $N=\{0\}.$
\end{proof}

\begin{remark}
 Locally compact abelian groups (LCA groups for short) have two useful properties: they are locally minimal and their structure is well known.
As to the second point, a  group is LCA iff it is topologically isomorphic to $\R^n\times H$ where $n\in\N_0$ and $H$ has a compact open
subgroup $K$.

An abelian group $G$ is dense in a LCA group iff it is locally precompact. Hence the local minimality criterion implies:
A locally precompact group $G$ is locally minimal iff it is locally essential in its completion $\wt{G}$; further, if
$\wt{G}=\R^n\times H$ with $n\in\N_0$ and $H$ having a compact open subgroup $K$, then $G$ is locally minimal if $G\cap (\{0\}^n\times K)$
is locally essential in $\{0\}^n\times K$, because the neighborhood $V$ which witnesses local essentiality may be chosen to be contained in
$[-1,1]^n\times K$. Any subgroup contained in $V$ is already contained in $\{0\}\times K$.
So the problem of local minimality of locally precompact groups is reduced to the question of precompact groups, which we will study in the following subsection.
\end{remark}

\subsection{First applications of the criterion of local minimality to dense subgroups of compact abelian groups}


\begin{lemma}\label{sandwich} Let $K$ be a compact abelian group, let $G$ be a subgroup of $K$ and let $N$ be a closed subgroup of $K$.

If $K/N$ is NSS then $G$ is locally essential in $K$ iff $G\cap N$ is locally essential in $N$.

\end{lemma}

\begin{proof}
According to item (a) of Proposition \ref{loc_ess}, we have to see only that if $G\cap N$  is locally essential in $N$  then $G$ is locally essential in $K$.
By assumption, there exists a neighborhood $V$ of zero in $K$ such that $V\cap N$ witnesses local essentiality of $G\cap N$ in $N$ and
such that $q(V)\subseteq K/N$ contains only the trivial subgroup, where $q:K\to K/N$ denotes
the canonical projection.
Let $L$ be a closed subgroup of $K$ contained in $V$ with $L\cap G=\{0\}$. Then $q(L)$ is a subgroup of
$K/N$ contained in $q(V)$. By the choice of $V$ this yields $q(L)=\{0\}$. Hence $L\subseteq \ker q =N$. Thus $L$ is a closed
subgroup of $N$ contained in $V\cap N$ with $L\cap G=\{0\}$. By our hypothesis this yields $L=\{0\}$.
\end{proof}


Our next aim is to study finite dimensional compact abelian groups.


\begin{remark}\label{RemarkAn}

If $K$ is a compact abelian group, the mappings
$$
N\mapsto N^\bot=\{\chi\in K^{\wedge}: \chi\restriction _N = 0\}\;\;\mbox{ and }\;\;H\mapsto H^\bot=\{x\in K: (\forall \chi \in H)\; \chi(x) =0\},
$$
define lattice antiisomorphisms (i.e., monotonely decreasing bijections) between the lattice of all closed subgroups $N$ of $K$ and the lattice of all subgroups $H$ of the discrete group $K^{\wedge}$, and they are inverse of one another.

In particular, $(N_1\cap N_2)^\bot =N_1^\bot + N_2^\bot$,
 for any closed subgroups $N_1,\; N_2$ of $K$. The subgroups $N^\bot$ and  $H^\bot$ are called annihilators of $N$ and $H$ respectively.

Recall further, that a compact abelian group is totally disconnected iff its character group is a
torsion group, and it is connected iff its character group is a torsion-free group.

If $N$ and $H$ are closed subgroups of $ K$ and $K^{\wedge}$ respectively, we have  $K/N\cong {(N^\bot)}^{\wedge}$
 and $H^\bot \cong (K^{\wedge}/H)^{\wedge} .$

\end{remark}
\begin{remark}\label{N_K}
Let $K$ be a finite-dimensional compact abelian group with $d = \dim K$. Then its Pontryagin dual $K^{\wedge}$
has free rank $d$ (see \cite[8.26]{HM}), so there exists a subgroup $H\cong \Z^d$ of  $K^{\wedge}$,
such that  $K^{\wedge}/H$ is torsion. The preceding remark implies that
 $H^{\bot}\cong ({K^{\wedge}/H})^{\wedge}$ is totally disconnected.

Let ${\mathcal N}_K$ denote the family of all closed subgroups $N$ of $K$ such that $K/N$ is isomorphic to
$\T^d$ (they  are totally disconnected as $\dim N = \dim K - \dim K/N =0$). In other words,
${\mathcal N}_K$ is the set of annihilators of subgroups $H$ of $K^{\wedge}$ which are torsion-free of rank $d$
(so that $K^{\wedge}/H$ is torsion in view of $r(H)=d=r(K^{\wedge})$).

For  $N,N'\in {\mathcal N}_K$,   also $N+N'\in {\mathcal N}_K$ and both $N$ and $N'$ have finite index in
$N+N'$. Indeed, if $H$ and $H'$ are the corresponding
annihilators, by Remark \ref{RemarkAn},  $K^{\wedge}/H$ and $K^{\wedge}/H'$ are torsion
  and the canonical mapping
$K^{\wedge}/(H\cap H')\to K^{\wedge}/H\times K^{\wedge}/H'$
is injective. This shows that $K^{\wedge}/(H\cap H')$ is  a torsion group and hence, by the rank-condition,
$H\cap H'\cong \Z^d$.
The character group of $(N+N')/N$ is isomorphic to $H/(H\cap H')$ and hence finite.

Obviously, the family of subgroups $H$  of $K^{\wedge}$ isomorphic to $\Z^d$ has no bottom element
(e.g., for each $H$ with this
property the {\em proper} subgroup $2H$ has the same property). Therefore, ${\mathcal N}_K$ has no top
element with respect to  inclusion.

In this circumstance, $N\cap N'$ need not belong to ${\mathcal N}_K$, but it is very close to, since $K/(N\cap
N')\cong \T^d\times F$, where $F$ is a finite abelian group.
Indeed, $(N\cap N')^\bot=H+H'$ is a finite extension of $H\cong \Z^d$.
When $K$ is also connected, then $K$ is a Lie group (i.e., isomorphic to $\T^d$) precisely when all $N\in
{\mathcal N}_K$ are
finite. In other words, the subgroups from the family $ {\mathcal N}_K$  help us to understand  ``how much"
$K$ differs from a Lie group.
\end{remark}

\begin{proposition}\label{corol0}
 Let $K$ be a finite-dimensional compact abelian group containing a dense subgroup $G$. Then TFAE:
\begin{itemize}
 \item[(a)] $G$ is locally minimal;
 \item[(b)] $G\cap N$ is locally essential in $N$ for some $N\in  {\mathcal N}_K$;
 \item[(c)]  $G\cap N$ is locally essential in $N$ for all $N\in  {\mathcal N}_K$.
\end{itemize}
\end{proposition}
\begin{proof}
(a) $\Rightarrow$ (c) is a consequence of Proposition \ref{loc_ess}(a) and the local minimality criterion,
while  (c) $\Rightarrow$ (b) is trivial. The implication
(b) $\Rightarrow$ (a) follows from Lemma \ref{sandwich}, since by the above considerations $K/N\cong \T^d$, and the local minimality criterion.
\end{proof}

\section{Compact abelian groups containing a ``small" dense locally minimal subgroup}

We start with a proposition which comes straightforward  from the results of Section 3.

\begin{proposition}\label{corol} Let $K$ be  a compact abelian group and $G$  a dense subgroup of  $K$.
\begin{itemize}
\item[(a)]
If $(G,\tau)$ is linearly topologized  then,  $G$ is locally minimal if and only if $G$ has an open minimal subgroup.

\item[(b)] If  for some infinite totally disconnected closed subgroup $N$  of $K$ one has $G\cap N=\{0\}$, then $G$ is not locally minimal.
\end{itemize}
    \end{proposition}

\begin{proof} (a) Note first that $K$ is also linearly topologized. Let $V$ be an open subgroup of $K$ witnessing local essentiality of $G$. Then
$G_1=G\cap V$ is an open subgroup of $G$ that is dense and essential in the compact group $V$. Hence $G_1$ is minimal. The converse is immediate.

(b) Note that $N$ is compact and infinite so  it cannot be discrete.
Since every compact totally disconnected group is linearly
topologized, we can apply Corollary \ref{Crit2}(c).
\end{proof}

\begin{corollary}\label{newtor}
Let $K$ be a compact abelian group and let $G$ be a subgroup of $K$. If  $K$ is torsion, then $G$ is locally minimal iff $G$ has an open subgroup that is minimal.
\end{corollary}
\begin{proof} Since any subgroup of a torsion group is torsion too, we may suppose without loss of generality that $G$ is dense in $K$.  It is  a consequence of \cite[Theorem (25.9)]{HR} that a compact torsion group has a linear  topology.
Now Proposition \ref{corol}(a) applies.
\end{proof}

\begin{remark}
Note that item (a) of Proposition \ref{corol} remains true without the assumption of precompactness; this follows easily from Lemma 2.3 in \cite{DM} (which we reproduce below as Proposition \ref{minsubgr}(a)). 

Let us note that compactness cannot be replaced by local compactness
in Proposition \ref{corol}(b). Indeed, $G=\Q$ is a dense locally minimal
subgroup of $K=\R$ (Corollary \ref{Crit3}) such  that $G\cap N=\{0\}$ for the closed and totally disconnected (actually,
discrete) subgroup $N=\langle \sqrt{2}\rangle$ of $K$.

\end{remark}

The following lemma, exhibiting the full power of Theorem \ref{crit}, will be used in all proofs of this subsection.

\begin{lemma}\label{corol1} Let $p$ be a prime number, let $K$ be a compact
abelian group and let $G$ be a dense locally minimal subgroup of $K$.
\begin{itemize}
\item[(a)]  If $K$ contains a (necessarily closed) subgroup  $N \cong \Z_p$, then $G\cap N \ne \{0\}$,
so $t(G)\subsetneq G$.
\item[(b)] If $K$ contains a (necessarily closed) subgroup isomorphic to $\Z_p^2$, then
$r_0(G)\geq \cont$.
\item[(c)] If $K$ contains a (necessarily closed) subgroup isomorphic to
           $\Z(p)^\omega$, then $|G|\geq \cont$.
\end{itemize}
\end{lemma}

\begin{proof}
According to Theorem \ref{crit}, $G$  is a locally essential subgroup of $K$ with respect to some \nbd\ $U$ of 0.

(a)  This follows directly from Proposition \ref{corol}(b).

(b) Let  $N$ be a closed subgroup of $K$ isomorphic to $\Z_p^2$. Since $|\Z_p|=\cont$, one has an enumeration $\{\xi_i: i\in I\} $ of $\Z_p\setminus \{0\}$, with $|I|= \cont$.
Following  the argument from \cite{P1} (see also \cite[Proposition 3.5.11]{DPS}), we introduce the groups $N_i=\{(\lambda, \xi_i\lambda):\lambda \in \Z_p \}$ for each $i\in I$. Note that $N_i \cong \Z_p$ and
$$
N_i\cap N_j=\{0\}\mbox{ for }i,j\in I , i\ne j\eqno(2)
$$
since the ring $\Z_p$ has no zero divisors. By item (a),  for every $i\in I$ there exists $0\ne g_i \in G\cap N_i$.  It follows from (2) that $g_i\ne g_j$
when $i\ne j$. Hence the subset $X=\{g_i: i\in I\}$ of $G$ has size $\cont$. So, it generates a
torsion-free subgroup $H$ of $G$ of size $\cont$. Therefore, $r_0(G)\geq r_0(H) =|H|=\cont$.

(c) Let  $N$ be a closed subgroup of $K$ isomorphic to $\Z(p)^\omega$. Applying Proposition \ref{loc_ess}(c), we conclude that there exists an open subgroup $V$ of $N$ such that $G'=G\cap V$ is an essential subgroup of $V$. Since $V$ must be isomorphic to $N$, from the fact that $G'=G\cap V$ is essential in  $V$ we can conclude that $G'= V$. Since the latter group has size $\cont$ we are done.
\end{proof}

\begin{corollary}\label{corol_1} Let $p$ be a prime number. If $K$ is a compact abelian group containing a (necessarily closed) subgroup isomorphic to $\Z_p^2$ or to
$\Z(p)^\omega$, then every dense
locally minimal subgroup of
$K$ has size $\geq \cont$.  \end{corollary}

\begin{remark}\label{pmonot} Recall that a compact group $N$ is called {\it $p$-monothetic} if $N$ is topologically isomorphic either to the $p$-adic integers $\Z_p$ or to $\Z(p^n)$ for some $n\ge 1$. In particular, every $p$-monothetic group has a closed subgroup which is topologically isomorphic to $\Z_p$ or $\Z(p).$ It is shown in \cite[4.1.7]{DPS} that every nontrivial compact abelian group has a nontrivial $p$-monothetic subgroup for some prime $p$.
 This allows us to take as test subgroups $N$ for essentiality only $p$-monothetic ones, or even only those subgroups topologically isomorphic to $\Z_p$ or $\Z(p)$ for some prime $p$.
\end{remark}
Next we show that under certain
restraints on the algebraic structure of a precompact
locally minimal group it turns out to be minimal. For example, the following corollary implies that if $K$ is a product of copies of ${\mathbb Z}_p$ with not necessary
distinct primes $p$, then a subgroup $G$ of $K$ is  locally minimal iff it is minimal.
\begin{corollary}\label{tor-free/tor}
Let $K$ be a compact abelian group and let $G$ be a subgroup of $K$.
If $K$ is torsion-free, then $G$ is locally minimal iff it is minimal.
\end{corollary}

\begin{proof} Since any subgroup of a torsion-free group is torsion-free  too, we may suppose without loss of generality that $G$ is dense in $K$.
 Let $G$ be locally minimal. In order to prove that $G$ is minimal,
  it is sufficient to verify that $ G$ is essential in $K$ (Theorem \ref{Min_Crit}). Taking into account Remark \ref{pmonot}, we fix a torsion-free $p$-monothetic group $N$. Obviously,  $N \cong \Z_p$.  Now Lemma \ref{corol1}(a) implies $ \{0\}\not= G \cap N$ which completes the proof.
\end{proof}

%
%
%
%
For precompact abelian groups $G$ of square-free exponent we have the following striking result:

\begin{corollary}\label{tor-free/tor2}
Let $G$ be a precompact abelian group of square-free exponent. The following are equivalent:
\begin{itemize}
 \item[(a)] $G$ is compact;
 \item[(b)] $G$ is minimal;
 \item[(c)] $G$ is locally minimal.
\end{itemize}
\end{corollary}

\begin{proof}
The implications (a) $\Rightarrow$ (b) $\Rightarrow$ (c) are obvious.

Let $K$ be the compact completion of $G$. Then $K$ has the same square-free exponent as $G$.
Since $G$ is dense in $K$, to prove that $G=K$ is compact it suffices to show that  $G$ is open (and consequently closed) in $K$.
It is a consequence of Corollary \ref{newtor} that  $G$ has an open, minimal subgroup $N$. The closure $N_1$ of $N$ in $K$ is
a compact group of square-free exponent, i.e., $Soc(N_1) = N_1$.
Since $N$ is essential in $N_1$, $N$ contains all elements of prime order of $N_1$. Consequently,
also $N_1= Soc(N_1)\subseteq N$. Hence $N$
is compact in $K$. Therefore $N$ is closed also in $K$. So it is open in $K$, as it is open in $G$.
\end{proof}

 We apply the criterion of local minimality in the following three theorems. The first two of them characterize the compact abelian groups containing a dense locally minimal subgroup that is ``small"  from some appropriate point of view. As a measure of ``smallness" we take either the size $|G|$ of the group, or its free-rank $r_0(G)$ (e.g., torsion groups $G$ are ``small" in the latter sense as $r_0(G)=0$ for a torsion abelian group). The third one characterizes the compact abelian groups whose torsion subgroup is dense and locally minimal. In these theorems the equivalences among the items (b)-(c)-(d) are the main theorems of \cite{P2}, \cite{D_Eger} and  \cite{DP}, respectively.


Let us recall here that every totally disconnected compact group $N$ can be written as a topological direct product $N=\prod_p N_p$ of
its {\em topologically $p$-torsion subgroups} $N_p=\{x\in N : p^nx \to 0\}$ (cf. \cite[Proposition 3.10]{Armacost}).


\begin{theorem}\label{theocorol1} For a compact abelian group  $K$ the following assertions are equivalent:
\begin{itemize}
  \item[(a)] $K$ contains a dense locally minimal subgroup of size $< \cont$;
  \item[(b)] $K$ contains a dense countable minimal subgroup;
  \item[(c)] $K$ contains subgroups isomorphic to $\Z_p^2$ or to $\Z(p)^\omega$ for no prime $p$.
  \item[(d)] $K$ is finite-dimensional and there exists a closed totally disconnected subgroup $N$ of $K$ such that:
\begin{itemize}
\item[(d$_1$)] $K/N\cong \T^d$, where $d=\dim K$;
\item[(d$_2$)] $N=\prod_p (\Z_p^{e_p}\times F_p)$, where $e_p\in \{0,1\}$ and $F_p$ is a finite $p$-group for each prime $p$.
\end{itemize}
\end{itemize}
\end{theorem}

\begin{proof} The implication (a) $\Rightarrow $ (c) follows directly from Corollary \ref{corol_1}, while (b) $\Rightarrow$ (a) is trivial. The equivalence among (b), (c) and (d) was proved in \cite{P2}.
\end{proof}

\begin{theorem}\label{corol2}
For a compact abelian group  $K$ the following assertions are
equivalent:
\begin{itemize}
  \item[(a)] $K$ contains a dense locally minimal subgroup of free rank
             $< \cont$;
  \item[(b)] $K$ contains a dense minimal subgroup of countable free rank;
  \item[(c)] $K$ contains
subgroups  isomorphic to  $\Z_p^2$ for no prime $p$.
  \item[(d)] $K$ is finite-dimensional and there exists a closed totally  disconnected subgroup $N$ of $K$ such that:
\begin{itemize}
\item[(d$_1$)] $K/N\cong \T^d$, where $d=\dim K$;
\item[(d$_2$)] $N=\prod_p N_p$, where $N_p=\Z_p^{e_p}\times B_p$,  $e_p\in \{0,1\}$ and $B_p$ is a compact $p$-group for each prime $p$.
\end{itemize}
\end{itemize}
\end{theorem}

\begin{proof} The implication (a) $\Rightarrow $ (c)  follows directly from  Lemma \ref{corol1}(b), while (b) $\Rightarrow$ (a) is trivial. The equivalence among (b), (c) and (d) was proved in \cite{D_Eger}.
\end{proof}

The compact abelian groups described in the next theorem
 are called {\em exotic tori} in \cite{DP}. The motivation for this name comes from the fact that every nontrivial closed subgroup of an exotic torus contains nontrivial torsion elements (the usual tori have obviously this property).

\begin{theorem}\label{ExTor} For a compact abelian group  $K$ the following
assertions
are equivalent:
\begin{itemize}
  \item[(a)] the subgroup $t(K)$ of $K$ is dense and locally minimal;
  \item[(b)] the subgroup $t(K)$ of $K$ is dense and minimal;
  \item[(c)] $K$ contains  subgroups  isomorphic to  $\Z_p$ for no  prime $p$.
  \item[(d)] $K$ is finite-dimensional and there exists a closed totally disconnected subgroup $N$ of $K$ such that:
\begin{itemize}
\item[(d$_1$)] $K/N\cong \T^d$, where $d=\dim K$;
\item[(d$_2$)] $N=\prod_p N_p$, where $N_p$ is a compact $p$-group for each prime $p$.
\end{itemize}
\end{itemize}
In case $K$ is a connected compact abelian group satisfying the above equivalent conditions, the subgroups $N_p$  in item (d$_2$) are finite.
\end{theorem}

\begin{proof} Obviously (b) $\Rightarrow$ (a).  To see that (a) implies (c) apply Lemma \ref{corol1} (a). The equivalence among (b), (c) and  (d) and the restriction to the case of connected $K$
 were proved in \cite{DP}.
\end{proof}

Although the same compact groups share the properties  from items (a) and (b) of the above theorem, 
%
%
%
%
%
%
 the torsion abelian groups $G$ that admit precompact locally minimal group topologies need not admit also a minimal group topology. An example is the Pr\"ufer group $\Z(p^{\infty}),$ which admits no minimal group topology (\cite{DM}; (3.5.4) in \cite{DPS}) but, when endowed with the topology induced by the usual one on $\T$, is precompact and locally minimal (Corollary \ref{Crit3}).

\begin{example} Let $p$ be a prime.
\begin{itemize}
  \item[(a)] It follows from Lemma \ref{corol1} (c) that the group $G=\bigoplus_\omega \Z(p)$  does not admit any
precompact locally minimal group topology. (The completion $K$ of $G$ is a compact group of exponent $p$, hence, according to \cite[Theorem (25.9)]{HR} topologically isomorphic to $\Z(p)^\kappa$ for some infinite cardinal $\kappa$. Now Lemma \ref{corol1}(c) applies.)
  \item[(b)] Let us see now that the group $G= \bigoplus_\omega \Z(p^\infty)$  does not admit any precompact locally minimal group topology either.
  Indeed, assume that $G$ carries such a topology. Then its completion $K$ is a compact abelian group. Moreover, since $G$ is divisible, $K$ is divisible as well, so
  $K$ is connected. Since $G$ is torsion, we conclude that  $G\leq t(K)$, so $t(K)$ is a dense and locally minimal subgroup of $K$ (Corollary \ref{Crit2}(a)). Now by Theorem \ref{ExTor} (d),
  $K$ is finite-dimensional and there exists a closed totally disconnected subgroup $N$ of $K$ such that $K/N\cong \T^d$, where $d=\dim K$ and $N=\prod_p N_p$,
  where $N_p$ is a finite $p$-group for each prime $p$.
Let $K_p$ denote the subgroup of all $p$-torsion elements of $K.$
Then $K_p/(N\cap K_p) = K_p/N_p$ is isomorphic to a subgroup of $t(\T^d).$
On the other hand, $N_p\cap G$ is a finite subgroup of $G$.
Since
$G/(G\cap N_p)$ is a divisible $p$-group of infinite $p$-rank (as $r_p(N_p) < \infty$), we conclude that $G/(G\cap N_p)\cong \bigoplus_\omega \Z(p^\infty)$. Since $G/(G\cap N_p)$ is isomorphic to a subgroup of $ t(\T ^d)$ and $r_p(\T^d) =d$, this yields $r_p(G/(G\cap N_p)) \leq d$, a contradiction.
\end{itemize}
\end{example}


\section{Locally quasi-convex groups}

In \cite{LocMin} we introduced the  class  of locally GTG groups. As we showed there, it fits very well in the setting of locally
minimal groups as it gives a nice connection between this class and that of minimal groups. Since  precompact abelian groups (\cite[Example 5.6(a)]{LocMin}), as well as UFSS groups (\cite[Example 5.2(a)]{LocMin}), are locally GTG, this explains the importance of this new
class. On the other hand, minimal abelian groups are precompact,  so minimal abelian groups are both locally minimal and locally GTG.
We proved in  \cite[Theorem 5.10]{LocMin} that a Hausdorff abelian topological group is UFSS iff it is locally minimal, NSS and locally GTG.
We are going to pursue this line in this section considering the class of locally  quasi-convex groups, which also contains the class of precompact abelian groups and, as we will see, it fits equally well in the setting of locally minimal groups.

The section starts with  the notions and some basic properties of UFSS groups, locally GTG groups and locally quasi-convex groups.
%

\subsection{UFSS groups, locally GTG groups and locally quasi-convex groups}

Recall that a  Hausdorff topological group $(G,\tau)$ is
{\em uniformly free from small subgroups} (UFSS for short) if for some symmetric
neighborhood $U$ of $0$, the sets $(1/n)U$ form a neighborhood basis at $0$ for $\tau$.
Neighborhoods $U$ satisfying this condition
will be said to be {\em distinguished.} It is easy to see that any symmetric neighborhood
of zero contained in a distinguished one is distinguished, as well.
One can see that a UFSS group $(G,\tau)$
with distinguished neighborhood $U$ has the following property, which trivially implies that $(G,\tau)$ is $U$-locally minimal: if
$\tcal$ is a group topology on $G$ such that $U$ is a $\tcal$-neighborhood  of $0$, then $\tau \leq \tcal$.
All UFSS groups are NSS groups.
A topological vector space is UFSS as a topological abelian group
if and only if it is locally bounded (see \cite[Example 3.2(c)]{LocMin}). In particular every normed space is a
UFSS group.

\begin{proposition}(\cite[Proposition 3.12]{LocMin})\label{perm_prop_3} The class of UFSS groups has the
following permanence properties:
\begin{itemize}
\item[(a)]  If $G$ is a dense subgroup of $\wt{G}$ and $G$ is  UFSS, then
$\wt{G}$ is UFSS.
\item[(b)] Every subgroup of a UFSS group is UFSS.
\item[(c)] Every group locally isomorphic  to a UFSS group  is UFSS.
\end{itemize}
\end{proposition}

 Let $G$ be  an abelian group and let $U$ be a symmetric subset of $G$ such that $0\in U.$ We say that $U$ is a {\em group topology
generating} subset of $G$ (``GTG subset of $G$" for short) if the sequence of subsets $\{(1/n)U \,:\,n\in \N\}$ is a basis of
neighborhoods of zero for a (not necessarily Hausdorff) group topology ${\mathcal T}_U$ on $G$.
A symmetric subset $U\subseteq G$ of an abelian group $G$
is a GTG subset if and only if
  $
  \exists m\in \N\ \mbox{with} \ (1/m)U + (1/m)U \subseteq U
  $ (see \cite[Proposition 4.4]{LocMin} ).

  A Hausdorff
topological abelian group $G$ is {\em locally GTG} if it admits a basis of neighborhoods of the identity formed by GTG subsets of $G$.

The {\em polar}  of  a subset $A \subseteq G$ is the set
$A^\triangleright : = \{\chi\in G^{\wedge}: \chi(A)\subseteq
\mathbb{T}_{+} \}$. The
{\em inverse polar} of a subset $B$ of $G^{\wedge}$ is the set
$B^\triangleleft : = \{x\in G: \chi(x)\in \mathbb{T}_{+}\; \forall
\chi \in B\}$. The {\em quasi-convex hull} of a subset   $A
\subseteq G$ is
$Q(A):=(A^\triangleright)^\triangleleft=\bigcap_{\chi\in
A^\triangleright }\chi^{-1}(\mathbb{T}_{+})$;   the set $A$ is said
to be {\it quasi-convex} if $Q(A)=A$ (i.e. if for every $x\in
G\setminus A$ there exists a  character $\chi\in G^{\wedge}$ such
that $\chi(A)\subseteq \T_+$ and $\chi(x)\notin \T_+$). When $A$ is
a subgroup of $G,$ its polar $A^\triangleright$ coincides with its
{\em annihilator} $A^{\perp}$, i. e. the subgroup of $G^{\wedge}$ formed by
those characters $\chi$ such that $\chi(A)=\{0\}$.

A topological abelian group is {\it locally quasi-convex} if it admits a basis of neighborhoods of zero formed by quasi-convex sets. Vilenkin was the first to define locally quasi-convex groups (\cite {vilenkin}), for groups with a boundedness. A natural example of a locally quasi-convex group is the underlying group of a locally convex space.  Moreover, a Hausdorff topological
vector space is locally quasi-convex as an (additive) topological abelian group if and only if it is a locally convex space
(\cite[2.4]{BTVS}). Subgroups and products of locally quasi-convex groups are locally quasi-convex. All character groups (endowed with
the compact open topology) are locally quasi-convex; in particular compact abelian (and thus precompact abelian) groups are locally quasi-convex, as duals of discrete abelian groups.

Quotient groups of locally quasi-convex groups are in general not locally quasi-convex. Even a quotient group of a locally
quasi-convex group with respect to a discrete subgroup need not be locally quasi-convex, despite being locally isomorphic to the group (see \cite{BB}).

It is shown in \cite[Example 3.5(b)]{LocMin} that  UFSS groups may fail to be maximally almost periodic. One may ask whether a maximally almost periodic UFSS group $G$ is already locally quasi-convex. A counterexample to this effect follows.

\begin{example}\label{lp}
Fix any $s\in (0,1)$ and consider the topological vector space $l^{s}$ of all sequences of real numbers which are summable in power $s$,
with the topology given by the following basis of neighborhoods of zero:
$$
U_r=\left\{(x_n)\in \mathbb R^{\mathbb N}\,:\,\sum_{n=1}^{\infty} |x_n|^s \le r\right\},\quad
r>0
$$
This space is not locally convex (see for instance \cite[Chapter 2]{Kal}), hence  the topological abelian group which underlies it is
not locally quasi-convex. 
Since $l^{s}$ is locally bounded (as $U_1\subseteq r^{-1/s}U_r$ for every $r>0$ )
it is  UFSS. On the other hand,  it is clear that the   characters of the form $x \mapsto q(\lambda
x_n)$ ($n\in {\mathbb N},\;\lambda \in {\mathbb R}$) separate the points of the group, where $q:\mathbb R \to \mathbb T$ is the canonical projection.
\end{example}

\begin{proposition}\label{Bemerkung}  Let $G$ be a topological abelian group.
\begin{itemize}
\item[(a)] 
For the canonical projection $q:\R \to \T$, $q([-1/4n,1/4n])=(1/n)\T_+$
\item[(b)] Let $U$ be a quasi-convex subset of $G$. For any $n\in {\mathbb N},$
$$
(1/n) U=\bigcap_{\chi\in U^\triangleright }\chi^{-1}((1/n)\T_+)= \bigcap_{\chi\in U^\triangleright}\chi^{-1}(q([-1/4n,1/4n])), \;  \text{and}\;
U_{\infty}=\bigcap_{\chi \in U^{\triangleright}}{\rm ker}\chi.
$$
All the sets $(1/n)U$ are quasi-convex. In particular, if $U$ is a quasi-convex neighborhood of $0$, so are the sets $(1/n)U$ for all $n\in\N$.
 \item[(c)] Let $U$ be a quasi-convex subset of $G$. For any $n\in {\mathbb N},$ $(1/2n)U+(1/2n)U\subseteq (1/n)U$. In particular $U$ is a GTG subset of $G$.
\item[(d)] If $G$ is locally quasi-convex, $G$ is locally GTG.
\end{itemize}
\end{proposition}
\begin{proof}
(a) is standard. The first part of (b) is straightforward, the quasi-convexity of $(1/n)U$ is a consequence of the fact that $(1/n)\T_+$ is quasi-convex.

 (c) Fix $x,y\in (1/2n)U$. According to (b), $\chi(x),\chi(y)\in q([-1/8n,1/8n])$  for all $\chi\in U^\triangleright$, which implies $\chi(x+y)\in q([-1/4n,1/4n])$  and hence, by (a),
  $x+y\in (1/n) U$.

 (d) is an immediate consequence of (c).
\end{proof}

The following proposition  allows us to find large, in appropriate sense, minimal subgroups in a locally minimal group.

\begin{proposition}\label{minsubgr} Let $G$ be a
$U$-locally minimal group.
\begin{itemize}\item[(a)] ({\rm \cite[Lemma 2.3]{DM}}) Let $H$ be a closed central subgroup
of $G$ such
that $H+V\subseteq U$ for some neighborhood $V$ of $0$ in $G$. Then $H$
is   minimal.
\item[(b)] Assume additionally that $U$ is a GTG set. Then
\begin{itemize}
\item[(b1)]  {\rm (\cite[ Theorem 5.12]{LocMin})} $U_\infty$ is a minimal subgroup.
\item[(b2)] {\rm (\cite[Corollary 5.15]{LocMin})} $(1/m)U=U_\infty$ is an open minimal subgroup provided that $G$ has finite exponent $m.$
\end{itemize}
\end{itemize}
\end{proposition}
For a topological abelian group $G$, let $\sigma(G,G^{\wedge})$ denote the initial topology on $G$ with respect to all continuous characters of $G$.
\begin{proposition}\label{injalmmin}
Let $(G,\tau)$ be a maximally almost periodic abelian topological group. If
$G$ is locally minimal with respect to a GTG neighborhood $U$, then $\tau$ is the supremum of the topologies $\sigma(G,G^\wedge)$ and ${\cal T}_U$. In particular $G$  is locally GTG.
\end{proposition}

\begin{proof}
The supremum topology $\sigma(G,G^{\wedge}) \vee {\cal T}_U$ is obviously coarser than  $\tau$. It is  well known that  $\sigma(G,G^{\wedge})$ is a Hausdorff group topology if and only if $G$ is maximally almost periodic. Hence $\sigma(G,G^{\wedge}) \vee {\cal T}_U$ is a Hausdorff topology. Obviously,  $U$ is a neighbourhood of $0$ in $\sigma(G,G^{\wedge}) \vee {\cal T}_U$.
So, since $G$ is $U$-locally minimal,
$ \sigma(G,G^{\wedge}) \vee {\cal T}_U=\tau$ holds. 
As the supremum of finitely many locally GTG group topologies, $\tau$ is locally GTG as well. \end{proof}


\subsection{Locally quasi-convex UFSS  group topologies}
In this subsection we continue the study of the algebraic structure of locally minimal groups we started in \cite[\S 5.2]{LocMin}. We show that every unbounded abelian group  can be endowed with a non-discrete locally quasi-convex UFSS group topology. This is of interest, since for no prime $p$ the group $\Z(p^\infty)$ admits a minimal group topology (\cite{DM}).

\begin{lemma}\label{ex_c_0}  For every increasing sequence $(m_n)$ of natural numbers the group $\bigoplus_{n\in\N}\Z(p^{m_n})$ admits
a non-discrete  locally quasi-convex UFSS and hence locally minimal group topology.
\end{lemma}

\begin{proof} Let us take the quotient of 
 the Banach space of null sequences $c_0$ with respect to the subgroup $D:=\langle e_n:\ n\in\N\rangle$ (where $e_n$ is the $n$-th unit
vector). Recall that $c_0:=\{(x_n)\in\R^{\N}:\ x_n\to 0\}$ is endowed  with the topology induced by the norm
$\|(x_n)\|_\infty:=\sup\{|x_n|:\ n\in\N\}$. In particular, $D$ is  discrete.  According to Proposition \ref{perm_prop_3}(b) and (c), every subgroup of  $c_0/D$ is
 UFSS and hence locally minimal. This holds in particular  for $G:=\langle p^{-m_n}e_n\rangle/D$. Since $p^{-m_n}e_n\to 0$ in
 $c_0$, the group $G$ is not discrete. Furthermore, $G$ is algebraically  isomorphic to $\bigoplus_{n\in\N}\Z(p^{m_n})$.

 Since subgroups of locally quasi-convex groups are locally  quasi-convex,  it is sufficient to show that $c_0/D$ is locally  quasi-convex.

Let $B$ be the closed unit ball of $c_0,$ and $\varphi:c_0\to c_0/D$ be the canonical epimorphism. We shall show that for every $m\in {\mathbb N}$, 
 $\varphi(\frac{1}{4m}B)$ coincides with its quasi-convex hull  $Q(\varphi(\frac{1}{4m}B))$.

Note that the dual of $c_0/D$ can be algebraically identified with the group ${\mathbb Z}_0^{\mathbb N}$ of all eventually zero integer sequences. Indeed, for any continuous character $\chi$ of $c_0/D,$ we have $\chi\circ \varphi \in c_0^{\wedge};$ according to \cite{Smi}, we may identify the dual group of $c_0$
with the dual space $c_0'$, and it is a well known fact that
$c_0'\to \ell^1,\ \varphi\mapsto (\varphi(e_n))_{n\in\N}$ is an (algebraic) isomorphism. (Recall that $\ell^1$ is the sequence space
of all absolutely summable sequences.) Hence, there exists $(\lambda_n)\in \ell^1$ such that
$\chi(\varphi(x))=q(\sum_{n=1}^{\infty}\lambda_n x_n)$ for every $x\in c_0;$ in particular $q(0)=\chi(\varphi(e_n))=q(\lambda_n)$ for
every $n\in {\mathbb N}$ and we deduce $(\lambda_n)\in {\mathbb Z}^{\mathbb N};$ since this sequence is in $\ell^1,$ it must be eventually null.

Fix $m\in {\mathbb N}.$ We will show that $(\varphi(\frac{1}{4m}B))^{\triangleright}=\{(k_n)\in {\mathbb Z}_0^{\mathbb N} :\sum|k_n|\le
m\}$. 
The inclusion $ ``\supseteq"{}$ is easy. Suppose now that $(k_n)\in {\mathbb Z}_0^{\mathbb N}$ is such that $\sum|k_n|> m;$ we will find $x\in
\frac{1}{4m}B$ with $q(\sum k_n x_n)\not \in q([-1/4,1/4]).$ For this, fix some $\delta \in {\mathbb R}$ with $1/4<\delta\le \min \{
\frac{1}{4m}\sum |k_n|,\,1/2\}. $ Put $$x_n=\frac{\delta\, {\rm sgn}(k_n)}{\sum|k_n|},\;n\in {\mathbb N}$$ It is clear that
$x=(x_n)\in \frac{1}{4m}B.$ On the other hand, $\sum k_n x_n=\delta\not \in [-1/4,1/4]+{\mathbb Z}.$

Now let us show that $\{(k_n):\sum|k_n|\le m\}^{\triangleleft}=\varphi(\frac{1}{4m}B).$ Again  the inclusion ``$\supseteq$"{}
is easy. For the reverse inclusion, fix $x\in c_0$ with $\varphi(x)\in \{(k_n):\sum|k_n|\le m\}^{\triangleleft}$; we need to
prove that $\varphi(x) \in \varphi(\frac{1}{4m}B).$ For every $k\in \{1,2,\cdots, m\}$ and every $j\in {\mathbb N}$ we have $ke_j \in
\{(k_n):\sum|k_n|\le m\}.$ Hence $kx_j\in [-1/4,1/4]+{\mathbb Z}$ for every such $k$ and $j.$ We deduce $x_j \in
[-\frac{1}{4m},\frac{1}{4m}]$ for every $j\in {\mathbb N}.$
\end{proof}

\begin{lemma} \label{chunk}
Let $G$ be an unbounded abelian group. Then $G$ contains a subgroup $H$ of one of the following forms:
\begin{itemize}
 \item[(i)] $H=\Z$,
 \item[(ii)] $H = \bigoplus_n\Z(p_n)$, where $p_n$ is an infinite set of  distinct primes;
 \item[(iii)] $H=\Z(p^\infty)$ for some prime $p$.
 \item[(iv)] $H = \bigoplus_n\Z(p^n)$ for some prime $p$.
\end{itemize}
\end{lemma}
\begin{proof} If $G$ is not a torsion group, it has a subgroup isomorphic to $\Z$ (case (i)); otherwise it is a torsion group and hence direct sum of its $p$-primary groups,
 i.e. $G=\bigoplus G_p$. So if for infinitely many $p$ the group $G_p$ is  not trivial, then case (ii) applies. Next we may assume that $G=G_p$.
 If $G$ is not reduced, then it has a subgroup isomorphic to $\Z (p^\infty)$  (case (iii)).

If $G=G_p$ is an infinite reduced abelian group, then fix a basic subgroup $B$ of $G$, i.e.,  $B$ is a pure subgroup  of $G$ (this means that $p^nB = B\cap p^n G$ for every natural $n$) that is a  direct sum of cyclic subgroups and $G/B$ is divisible \cite[32.2]{Fuc}.  It suffices to  see that $B$ is unbounded; then clearly $B$ will   contain a subgroup $H$ as in (iv). Assume  for a contradiction that $p^nB=0$.  Then divisibility of $G/B$ yields $p^nG+B=G$. Since $B\cap p^n G= p^nB = 0$, this sum is direct. Then $p^nG\cong G/B$ is
divisible, hence  zero as $G$ is reduced. So $G=B$ is bounded, a contradiction.
\end{proof}
\begin{theorem}\label{Ex_Lydia}  Let $G$ be an abelian group. Then the following assertions are equivalent:
\begin{itemize}
 \item[(a)] no subgroup of $G$ of size $<\cont$ admits a non-discrete locally  minimal and locally GTG group topology;
 \item[(b)] no countable subgroup of $G$ admits a non-discrete  locally quasi-convex UFSS group topology;
 \item[(c)] $G$ is bounded.
\end{itemize}
\end{theorem}

\begin{proof} The implication (a) $\Rightarrow$ (b)    is obvious as UFSS groups are locally minimal and  locally GTG.

(b)  $\Rightarrow $ (c): Suppose that $G$ is not bounded. Let us show that there exists a non-discrete locally quasi-convex UFSS group  topology on any subgroup $H$ of each of the forms (i) to (iv) in Lemma \ref{chunk}.

In cases (i) - (iii) the group $H$  can be embedded (algebraically) in $\T$, in such a way that $H$ is dense in ${\mathbb T}$. Since
endowed with the usual topology, $\T$ is a UFSS group, $H$ can be endowed with a UFSS group topology according to Proposition \ref{perm_prop_3}(b), which is not discrete. In case (iv), this is a consequence of Lemma \ref{ex_c_0}.

For the proof of (c) $\Rightarrow$ (a), let $G_0$ be a subgroup of $G$ of size $<\cont$. We can assume without loss of generality that $G_0$ is infinite. Let
$\tau$ be a locally GTG  and  locally minimal group topology on $G_0$. According to Proposition \ref{minsubgr}(b2), $G_0$ has an open, minimal subgroup $H$, which is also bounded.

Let us see first that $H$ is finite. Indeed, $H$ is a minimal abelian group of  size $<\cont$. Hence $r_p(H)<\infty$ for all primes $p$
(see \cite[Corollary 5.1.5]{DPS}). Since $H$ is  a bounded abelian group we conclude that  $H$ is finite. Since $H$ is open,  $\tau$ is discrete.
\end{proof}

\begin{corollary} Every unbounded abelian group admits a non-discrete locally quasi-convex UFSS group topology.
\end{corollary}
\begin{proof} Let $G$ be an unbounded abelian group. By Theorem \ref{Ex_Lydia}, some subgroup $H\le G$ admits a non-discrete locally quasi-convex UFSS group topology. This topology can be extended to a non-discrete locally quasi-convex UFSS group  topology on $G$ by declaring $H$ to be open in $G$.
\end{proof}
\begin{remark} In \cite[Corollary 4.25]{LocMin} it was proved, using completely different methods, that an abelian group is unbounded iff it admits a
non-discrete UFSS group topology.
\end{remark}

\section{Almost minimal groups}

\subsection{The notion of almost minimal group}

Our aim is to replace local minimality by a stronger property that still covers local compactness, UFSS and minimality (in the abelian case), but goes closer to them in
the following natural sense:

\begin{definition} A topological group $G$ is called {\em almost minimal} if it has a closed, minimal normal subgroup $N$ such that the quotient group $G/N$ is UFSS.
\end{definition}

\begin{remark}   Note that, with the above notations, $N$ is a $G_\delta$-set in $G$ (since $N$ is the inverse image
of the neutral element in the metrizable group $G/N$).

Moreover, if $N$ is also abelian then the Prodanov-Stoyanov Theorem implies that $N$ is precompact. \end{remark}

\begin{example}
Minimal groups, as well as UFSS groups, are almost minimal.
\end{example}

\begin{lemma}({\bf Merzon's Lemma}, \cite[Lemma 7.2.3]{DPS})\label{Merzon}
Let $H$ be a subgroup of a group $G$ and let $\tau$, $\sigma$ be  two (not necessarily Hausdorff) group topologies on $G$ such that
$\tau\leq\sigma$, $\tau_{|H}=\sigma_{|H} $, and  the topologies on the quotient group $G/H$  satisfy $\ol{\tau} =\ol{\sigma}$. Then $\tau=\sigma$. \end{lemma}

Almost minimal groups are locally minimal. In fact

\begin{theorem}\label{char_alm_min} For a topological abelian group $G$ the following conditions are equivalent:
\begin{itemize}
\item[(a)] $G$ is almost minimal.
\item[(b)] There exists a GTG neighborhood of zero $U$ such that $G$ is $U$-locally minimal and $G/U_{\infty}$ is UFSS.
\item[(c)] There exists a GTG neighborhood of zero $U$ such that $G$ is $U$-locally minimal and for every $V\in {\mathcal V}(0)$ there
exist $n\in {\mathbb N}$ and a finite $F\subseteq U_{\infty}$ with $(1/n)U\subseteq F+V.$
\end{itemize}
\end{theorem}

\begin{proof}
(a)$\Rightarrow$(c): Let $(G,\tau)$ be an almost minimal topological group with closed minimal subgroup $N$ such that $G/N$ is UFSS with
distinguished neighborhood $W$. This means in particular, that the quotient topology $\ol{\tau}$ coincides with ${\cal T}_{W}$.

Let $\pi:G\to G/N$ denote the canonical projection and let $\sigma$ be a Hausdorff group topology on $G$ coarser than $\tau$ such that $U:=\pi^{-1}(W)   \in{\cal V}_\sigma(0)$. Since inverse images of GTG sets by group homomorphisms are GTG (see \cite[Lemma 4.9(a)]{LocMin}), $U$ is a GTG set. Moreover $U_\infty=\pi^{-1}(W_\infty)=N$.

Since $N$ is minimal, we have $\sigma|_N=\tau|_N$.

We show next that the quotient topologies $\ol{\sigma}$ and $\ol{\tau}$ on  $G/N$ coincide. To this end we note first that $\ol {\sigma}\leq \ol {\tau}$ as a consequence of
$\sigma \leq \tau$ and $\pi(U)=W$ is a $\ol{\sigma}$-neighborhood of  $0_{G/N}$ hence ${\cal T}_{W}\leq\ol{\sigma}\leq \ol{\tau}={\cal T}_W$,  which
shows that $\ol{\sigma}=\ol{\tau}.$  Now we can apply Merzon's Lemma \  to obtain that $\sigma=\tau$ and hence $G$ is $U$-locally minimal.

Let $V$ be a neighborhood of zero in $G$. Fix another zero neighborhood $V'$ such that $V'+V'\subseteq V.$ There exists $n\in {\mathbb N}$ with $(1/n)W\subseteq \pi(V').$ Hence
$$
(1/n)U= (1/n)\pi^{-1}(W)=\pi^{-1}((1/n)W)\subseteq \pi^{-1} \pi(V')=V'+U_{\infty}.
$$
Now, since $U_{\infty}$ is precompact by Prodanov-Stoyanov's theorem, there exists a finite $F\subseteq U_{\infty}$ with
$U_{\infty}\subseteq F+V'$. Thus $(1/n)U\subseteq F+V'+V'\subseteq F+V.$

(c)$\Rightarrow$(b): Fix $U$ as in (c). Let us see that $G/U_{\infty}$ is UFSS. Since $U$ is GTG, there exists $m\in
{\mathbb N}$ with $(1/m)U+(1/m)U\subseteq U.$ We will prove that $W:=\pi((1/m)U)$ is a distinguished neighborhood of zero in
$G/U_{\infty}.$ In order to do this, fix $V\in {\mathcal V}(0).$ By hypothesis there exist $n\in {\mathbb N}$ and a finite $F\subseteq
U_{\infty}$ such that $(1/n)U\subseteq F+V.$ In particular $(1/n)U\subseteq U_{\infty}+V,$ and we deduce
$(1/n)((1/m)U+U_{\infty})\subseteq (1/n)U\subseteq U_{\infty}+V.$
Now, $(1/m)U+U_{\infty}=\pi^{-1}\pi((1/m)U)=\pi^{-1}(W),$ hence
$\pi^{-1}((1/n)W)=(1/n)(\pi^{-1}(W))\subseteq U_{\infty}+V,$ so $(1/n)W\subseteq \pi(U_{\infty}+V)=\pi(V).$

(b)$\Rightarrow$(a): This is an immediate consequence of the fact that if $G$ is a  $U$-locally minimal abelian group where $U$ is a GTG set, then $U_\infty$ is a minimal subgroup (see Proposition \ref{minsubgr}(b1)).
\end{proof}

It is not difficult to show that if $G$ is almost minimal, a convenient $U\in{\mathcal V}(0)$ can be chosen which simultaneously satisfies (b) and (c). However, in order to avoid ambiguities, in what follows we will use the expressions {\em $U$ witnesses almost minimality of $G$} or {\em  $G$ is $U$-almost minimal} as substitutes for condition (b).

\begin{proposition}\label{map_gtg}
Every maximally almost periodic, almost minimal group is locally GTG.
\end{proposition}

\begin{proof} By Theorem \ref{char_alm_min} (b), there exists
a GTG neighbourhood $U$ such that $G$ is $U$-locally minimal. The assertion is now a direct  consequence of Proposition \ref{injalmmin}.
\end{proof}

\begin{proposition} \label{completion_am}
Let $H$ be a dense subgroup of an abelian topological group $G$. If $H$ is almost minimal then $G$ is almost minimal, too.
\end{proposition}

\begin{proof} Let $N$ be a closed minimal normal subgroup of $H$ such that $H/N$ is UFSS.
 According to the Grant-Sulley Lemma (see the proof in \cite[Lemma 4.3.2]{DPS}), $H/N\rightarrow G/\ol{N}$ is an embedding with dense image. 
 Because of the minimality criterion (Theorem \ref{Min_Crit}), $\ol{N}$ is a minimal and (trivially) closed subgroup of $G$. Since $G/\ol{N}$  has a dense UFSS subgroup, it is UFSS (see Proposition \ref{perm_prop_3}(a)).
\end{proof}

\begin{remark} Complete almost minimal abelian groups have some interesting properties that are worth mentioning. Indeed, such a group $G$ has a compact
subgroup $N$ such that $G/N$ is UFSS, in particular metrizable. Hence $G$ is almost metrizable so \v Cech complete (being complete). This implies that $G$ is a $k$-space and also a Baire space.
\end{remark}

In the remaining part of this section we show that every locally quasi-convex locally minimal group can be embedded into an almost minimal group. The proof makes use of some properties of free abelian topological groups which are reminded beforehand. Moreover, by applying an open mapping theorem to this embedding, we show that every complete locally quasi-convex locally minimal group is already almost minimal.

We start with a typical example of a locally quasi-convex UFSS group that we are going to need in what follows:

\begin{example}\label{ckt} For a compact space $K$, the group ${\cal C}(K,\T)$ of continuous functions $K\rightarrow \T$ endowed with the topology
of uniform convergence is a locally quasi-convex UFSS group with distinguished neighborhood $U=\{f\in{\cal C}(K,\T)|\ f(K)\subseteq
\T_+\}$; actually, $(1/n)U=\{f\in{\cal C}(K,\T)|\ f(K)\subseteq (1/n)\T_+\}$. We are going to deal with the following particular case: $K$ is the polar $V^\triangleright$
of a neighborhood $V\in {\cal V}(0)$ in an abelian Hausdorff group. The fact that this set is compact is a standard one; see e.g.~\cite[Proposition 3.5]{Diss}.\end{example}

Locally quasi-convex UFSS groups form the group analog of normed spaces. Since every normed space can be embedded in a Banach space
of the form ${\cal C}(K)$ (the continuous functions of a suitable compact space $K$ into the field) it is natural to ask whether locally quasi-convex UFSS groups have a similar property.

\begin{proposition} Let $(G,\tau)$ be a locally quasi-convex UFSS group with distinguished quasi-convex neighborhood $V$. Then
$$
\beta:(G,\tau)\longrightarrow {\cal C}(V^\triangleright,\T),\ x\longmapsto \alpha_G(x)|_{V^\triangleright}
$$
is an embedding.
\end{proposition}

\begin{proof} It is obvious that $\beta $ is a homomorphism.  For $x\in G$ and
$\chi\in V^\triangleright$,
$\chi(x)=\alpha_G(x)(\chi)=\beta(x)(\chi)$, then by Proposition \ref{Bemerkung}
(b),
$$
(1/n)V=\{x\in G|\ \chi(x)\in (1/n)\T_+\ \forall\chi\in V^\triangleright\},
$$
which implies $\beta((1/n)V)={\rm im}\beta\ \cap \ \{f\in{\cal C} (V^\triangleright,\T)|\ f(V^\triangleright)\subseteq (1/n)\T_+\}.$ This shows that $\beta$ is an embedding.
\end{proof}

\begin{facts}\label{facts}
\begin{itemize}
\item[(A)] For every Tychonoff space $X$, the free abelian topological group $A(X)$ exists and is characterized by the following properties:
there exists an embedding  $\eta:X\rightarrow A(X)$ such that $\eta(X)$ is a basis of $A(X)$ and for every continuous mapping
$f:X\rightarrow H$ where $H$ is an abelian topological group there exists a continuous homomorphism $F:A(X)\rightarrow H$ which satisfies $F\circ\eta=f$.
(\cite{Markov})

\item[(B)] For  a compact Hausdorff space $K$, the free abelian topological group $A(K)$ is a hemicompact $k$-space, and a cobasis for the
compact sets is given by the family of sets $\{\eta(K)+\buildrel n \over \dots +\eta(K)\,:\,n\in\N\}$ (\cite{MMO}).
\item[(C)] For a compact Hausdorff space $K$, the dual group of $A(K)$ is topologically isomorphic to ${\cal C}(K,\T)$; more precisely,
$A(K)^\wedge \rightarrow {\cal C}(K,\T),\ \chi\mapsto \chi\circ \eta$ is a topological isomorphism   (\cite{Pestov}). In particular $A(K)^\wedge$ is a UFSS group (according to Example \ref{ckt}).
\item[(D)] A quotient  mapping of a hemicompact $k$-space onto a Hausdorff space is   compact covering (\cite{Morita}).
\end{itemize}
\end{facts}

\begin{theorem} \label{embedding}Let $(G,\tau)$ be a locally quasi-convex locally minimal topological group. Then $(G,\tau)$ can be embedded
into a complete locally quasi-convex topological group $A$ such that the latter group has a compact subgroup $B$ such that $A/B$ is
UFSS and $N:=G\cap B$ is a minimal subgroup of $G$. In particular, $A$ is an almost minimal group.
\end{theorem}

\begin{proof} Let $U$ be  a quasi-convex neighborhood of $0$ in $G$ which witnesses local minimality. According to Proposition \ref{minsubgr}(b1), $N:=U_\infty$ is
minimal. Further, $K:=U^\triangleright$ is a compact subset of $G^{\wedge}$. Let $\Phi:A(K)\rightarrow G^{\wedge}$ be the
continuous homomorphism which extends the embedding $\phi:K\to
G^{\wedge}$ and let $\rho$ be the final topology on $G^{\wedge}$ induced by $\Phi $. $H:=\Phi(A(K))=\langle U^\triangleright\rangle$
is an open subgroup of  $(G^{\wedge},\rho)$. Since $G^{\wedge}$ (with the compact-open topology)
is a Hausdorff group and  $\rho$ is a finer topology, $(G^{\wedge},\rho)$ is a Hausdorff  group, too. Since $\psi:A(K)\to (H,\rho|_H),
\ x\mapsto \Phi(x)$ is a quotient mapping, Facts \ref{facts} (B) and (D) imply that $\psi$ is compact  covering.   Since $N=U_\infty\subseteq U$,
we have $U^\triangleright\subseteq N^\triangleright$ and since the latter set is a subgroup, $H=\langle U^\triangleright\rangle \subseteq N^\triangleright$.

Let $\sigma$ be the group topology on $G$ of uniform convergence on the compact subsets of $(G^{\wedge},\rho)$ this means, a
neighborhood basis of $0$ is given by the sets $(S^\triangleleft)$ where $S$ runs through all compact subsets of $(G^{\wedge},\rho)$.

We shall show that

\begin{itemize}
\item[(a)] $U$ is a neighborhood of $0$ in $(G,\sigma)$.
\item[(b)] $(G,\sigma)$ is a Hausdorff topological group.
\item[(c)] $\sigma\leq \tau$.
\end{itemize}

(a) Since $U$ is quasi-convex, we have $U=U^{\triangleright\triangleleft}$. It remains to observe that $
U^{\triangleright}=K$ is  a compact subset of $(G^{\wedge},\rho)$; indeed,   $\eta: K \rightarrow A(K)$   and $\Phi$ are continuous.

(b) This is clear, since $G$ is maximally almost periodic.

(c) Let $S$ be a compact subset of $(G^{\wedge},\rho)$. By the definition of $\rho$
and since $H$ is open in $(G^{\wedge},\rho)$, there is a compact subset $S_0$ contained in $(H,\rho|_H)$ and a finite set $\{\chi_1,\ldots,\chi_n\}$ in $G^{\wedge}$ such that $S\subseteq \bigcup_{j=1}^n(\chi_j + S_0)$. Hence
$$
S^\triangleleft \supseteq \bigcap_{j=1}^n(\chi_j + S_0)^\triangleleft\supseteq (1/2)S_0^\triangleleft\cap (1/2)\{\chi_1,\ldots,\chi_n\}^\triangleleft.$$

Since $\psi$ is compact covering, there exists $m\in\N$ such that $\psi(\eta(K)+\buildrel m \over \dots +\eta(K))=(K+\buildrel m \over \dots+K)\supseteq S_0$. Now Proposition \ref{Bemerkung}(b) implies that   $(1/m)U\subseteq (K+\buildrel m \over \dots+K)^\triangleleft\subseteq
S_0^\triangleleft$ and hence $(1/2m)U\subseteq (1/2)S_0^\triangleleft$. Since $(1/2)\{\chi_1,\ldots,\chi_n\}^\triangleleft$ is a neighborhood
of
$0$ in $\tau$ it follows that $\sigma$ is coarser than $\tau$.

By assumption, $(G,\tau)$ is locally minimal with respect to $U$, so (a)-(c) imply  that $\sigma=\tau$.

Hence
$$
\beta:G\longrightarrow {(G^{\wedge},\rho)^{\wedge}},\
x\longmapsto(\chi\mapsto \chi(x))
$$
is an embedding where the character group is (as usual) endowed with the compact-open topology. Since $H$ is open in
$(G^{\wedge},\rho)$, $H^\triangleright $ is a compact subgroup of $(G^{\wedge},\rho)^{\wedge}$. We define $A:=(G^{\wedge},\rho)^{\wedge}$ and $B:=H^\triangleright$. $A$ is complete. Indeed, $A$ is the character group of $(G^{\wedge},\rho),$ which has an open hemicompact subgroup. Hence the assertion follows from Proposition 7.1 in \cite{glockner}.

Then,   $A/B\cong H^{\wedge}$ (see \cite[Lemma 2.2]{BCM}) and since $\psi:A(K)\rightarrow H$ is compact covering, the dual homomorphism
is an embedding. Therefore, by  Facts \ref{facts}(C) , $H^{\wedge}$ is topologically isomorphic to a subgroup of ${\cal C}(K,\T)$ and hence UFSS.  From
$$
\beta^{-1}(B)=\beta^{-1}(H^\triangleright)=\{x\in G|\ \beta(x)\in H^\triangleright\}=
\{x\in G|\ \chi(x)=0\ \forall
\chi\in U^\triangleright\}=\bigcap_{\chi\in U^\triangleright}{\rm ker}(\chi)\stackrel{{\rm Prop.~}\ref{Bemerkung}(b)}{ = }U_\infty
$$
the assertion $N=\beta^{-1}(B)$ follows.
\end{proof}

\begin{lemma}\label{HomSatz}[Open Mapping Theorem] Let $G$ be  a closed and $B$ a compact subgroup of the abelian Hausdorff group $(A,\tau)$.
Let $N:=B\cap G$. Then $\iota:G/N\rightarrow A/B,\ x+N\mapsto x+B$ is an embedding.
\end{lemma}

\begin{proof} Note first that the canonical homomorphism $f: A \to A/B$ is closed.  Indeed, if $F\subseteq A$ is closed, then
$F+B=f^{-1}(f(F))$ is still closed as $B$ is compact. Since $f$ is an open map, this implies that $f(F)$ is closed. In particular, the
image $f(G)$
%
%
of $G$ under $f$ is closed in $A/B$. Since $\iota:G/N\rightarrow A/B$ is a continuous injection, it suffices
only to check that $\iota$ is closed. But this is obvious, since $f\restriction _G:G\to f(G)$ is a closed map (as $G$ is closed in
$A$) and coincides with the composition of $\iota$ and  the surjection $G\rightarrow G/N$.
\end{proof}

\begin{corollary}\label{completion_slm} Every complete locally quasi-convex locally minimal abelian group is almost minimal. In particular, every LCA group is almost minimal. \end{corollary}

\begin{proof} Let $G$ be a complete locally quasi-convex locally minimal abelian group. According to Theorem \ref{embedding}, $G$ can be embedded into  an almost minimal group $A$ which has a compact subgroup $B$ such that $A/B$ is UFSS and $N:=G\cap B$ is minimal. Since $G$ is complete, its image in $A$ under the embedding is closed. According to Lemma \ref{HomSatz}, $G/N\rightarrow A/B$ is an embedding. Since subgroups of UFSS groups are UFSS (Proposition \ref{perm_prop_3}(b)), the assertion follows.
\end{proof}

The following generalization of Corollary \ref{completion_slm} is based on the fact that forming completions preserves both local minimality (Corollary \ref{Crit2}(a)) and local quasi-convexity (this is Lemma 8 in \cite{elenawojtek}; see also Corollary 6.17 in \cite{Diss}):

\begin{corollary} \label{completion_slm2}
The completion of a locally quasi-convex locally minimal abelian group is almost minimal.
\end{corollary}

However, a locally quasi-convex (even a precompact)
locally minimal abelian group need not be almost minimal, as we will see in Example \ref{counterexample} below. Before presenting this example, we give a positive result which pushes the criterion of local minimality to a higher level, thus identifying the almost minimal groups among all
precompact locally minimal groups.

\begin{proposition}\label{Rem_Str_Loc_Min}  If $G$ is an abelian precompact locally minimal group,
then $G$ has a minimal closed subgroup $N$ such that $G/N$ admits a continuous monomorphism $G/N \hookrightarrow \T^n$ for some $n\in \N$.
\end{proposition}

\begin{proof} Let $G$ be $U$-locally minimal, where $U$ is a quasi-convex neighborhood of $0$. The closed subgroup $N=U_\infty$ is minimal (Proposition \ref{minsubgr}(b)), the quotient $G/N$ is precompact as
well, and it admits the coarser UFSS (and precompact) topology induced by ${\cal T}_U$. The completion $K=\wt{G/N}$, with respect to that
coarser UFSS and precompact topology, is a compact UFSS group according to Proposition \ref{perm_prop_3}(a). It is a well known fact (see for instance \cite[32.1]{stroppel}) that
every compact UFSS group  is a Lie group.
This implies that that $K$ can be embedded in ${\mathbb T}^n$ for suitable $n\in {\mathbb N}$.
\end{proof}

As we have already said, this ``smallness"{} of the quotient $G/N$ with respect to the closed minimal subgroup $N$ does not imply that all precompact locally minimal groups are almost minimal, as the example given in the next subsection shows.

\subsection{A locally minimal, metrizable, precompact abelian group need not be almost minimal}

\begin{example}\label{counterexample}
There exists a locally minimal, metrizable, precompact abelian
group which is not almost minimal. More concretely, we will find a metrizable precompact locally minimal abelian group $G$ which has a closed subgroup $N$ such that $G/N$ is not UFSS and any minimal subgroup of $G$ is contained in $N$. Hence, if $N'$ is any closed minimal subgroup of $G$, there exists a continuous epimorphism $G/N'\to G/N.$ Since  every continuous homomorphic image
of a precompact UFSS group is UFSS (see \cite[Corollary 3.15]{LocMin}),  this  implies that $G/N'$ cannot be UFSS.

Let us fix some notations: We consider the subgroup $\Z$ of the discrete group $\Q$. As
$$
\Z^\bot\cong (\Q/\Z)^{\wedge} \cong (\bigoplus_{p\in\Prm}{\Z(p^\infty)})^{\wedge}\cong \prod_{p\in\Prm}\Z_p,
$$
and the group of $p$-adic integers is monothetic for every $p$ (topologically generated by $1_p\in\Z_p$), we can fix an isomorphism $\Z^\bot\to
\prod_{p\in\Prm}\Z_p$ and select elements $a_p\in \Z^\bot$ corresponding to $1_p$. Making use of this isomorphism we write
$\Z^\bot=\prod_{p\in \Prm}{\mathbb H}_p$ where ${\mathbb H}_p=\ol{\langle a_p\rangle}$. Let
$$
h:{\Q}^{\wedge}\longrightarrow {\Q}^{\wedge}/\prod_{p\in\Prm}p^2{\mathbb H}_p=:K
$$
be the canonical projection. The group $K$ is divisible, monothetic, compact and connected
(as $\Q^{\wedge}$ has these properties, see \cite[24.25, 24.32]{HR}). This will be tacitly used in the sequel. Let $x$ be a topological generator of $K$.

We consider the following subgroups of $K$:
$$
\begin{array}{rclcl} C_{p^2}&\quad:=\quad&\langle h(a_p)\rangle&\quad\cong\quad & \Z(p^2)\\
N&\quad:=\quad&\bigoplus_{p>2}pC_{p^2}&&\\
G&\quad:=\quad&\langle x\rangle \oplus N&&\\
L_0&\quad:=\quad&\prod_{p>2}C_{p^2}&&\\
L&\quad:=\quad&\prod_{p\in\Prm}C_{p^2}=h(\Z^\bot)&\quad\ \cong\quad&
\Z(4)\times L_0
\end{array}
$$

The precompact group $G$ is  the desired example of a locally minimal group that is not almost minimal. In order to prove this, we will make use of the following properties:
\begin{itemize}
\item [(i)] $N=G\cap L_0$

\item[] "$\subseteq$" is easy. Conversely, fix $kx+t\in G\cap L_0$ where
$t\in N$ and $k\in \Z$. It follows that $kx\in L_0$. Suppose
$k\not=0$. The divisibility of $K$ implies that
$$K=kK=k\ol{\langle x\rangle}\subseteq \ol{\langle
kx\rangle}\subseteq L_0,$$ which is a contradiction.
\item[] Items (ii) to (vi) are obvious:
\item[(ii)] $G$ is dense in $K$ (since ($K=\ol{\langle x\rangle}$) and in particular it is precompact.
\item[(iii)] $\ol{N} =\prod_{p>2} pC_{p^2} \leq L_0$,
\item[(iv)] $L_0/\ol{N}\cong \prod_{p>2}\Z(p) \cong \ol{N}$,
\item[(v)] $ L/\ol{N}\cong \Z(4)\times \ol{N}.$
\item[(vi)] $K/L\cong h(\Q^{\wedge})/h(\Z^\bot)\cong \Q^{\wedge}/\Z^\bot\cong \T$.
\item[(vii)] $K/L_0\cong \T. $
\item [] In order  to prove (vii), observe  that the connected group $K/L_0$ has a subgroup $L/L_0\cong C_{2^2}\cong \Z(4)$ (by the definitions) and further
 $(K/L_0)/(L/L_0)\cong K/L\cong \T$ (by (vi)). Let $X$ denote the dual group  of $K/L_0$. According to \cite[24.25]{HR}, $X$ is
 torsion-free  and $X$ has a subgroup
$ (L/L_0)^\perp\cong ((K/L_0)/(L/L_0))^{\wedge}\cong (K/L)^{\wedge}\cong  \T^{\wedge} \cong \Z$ of index $4$, since $(K/L_0)^{\wedge}/(L/L_0)^\bot$ $\cong
(L/L_0)^{\wedge}\cong \Z(4)$.  Therefore, $X\cong \Z $ too.

Now the formula  (vii)  provides a continuous surjective homomorphism $\zeta: K\to   \T$ with kernel $L_0$. Let $W= \zeta ^{-1}(\mbox{Int}(\T_+))$. Since inverse image of GTG sets by group homomorphisms are GTG sets (see \cite[Lemma 4.9]{LocMin}),   $W$ is a GTG-neighborhood of 0 in $K$ such that $W_\infty=\zeta^{-1}(\mbox{Int}(\T_+)_\infty)= {\rm ker}(\zeta)=L_0$.
\item[(viii)] Every closed nontrivial subgroup $S$ of $K$ contained in $W$ is a subgroup of $L_0$ and has nontrivial intersection with $N$ and hence with $G$.
\item[] Observe first  that   $S$ must be contained in $L_0$, the largest subgroup contained in $W$. Now we prove that every such $S$
nontrivially meets $N$. It is a well known fact that a closed subgroup $S$ of $L_0$ must have the form
$S= \prod_{p>2} A_p$, where each $A_p $ is a subgroup of the
respective $C_{p^2}$ (\cite[Example 4.1.3]{DPS}). Clearly, $A_p\ne 0$ for at least  one $p$ as $S\ne 0$. Then $ 0\ne pC_{p^2} \subseteq S\cap N\subseteq  S \cap G. $

\item[(ix)] $G$ is locally minimal.

\item[] This follows from Theorem \ref{crit}, since  by (viii) $G$ is locally essential in $K$, as witnessed by $W$.

\item[(x)] $N$ is closed in $G$.
\item[] This is clear, since $N=G \cap L_0$ and $L_0$ is closed in $K$.
\item[(xi)]  Every proper closed subgroup $H$ of $G$ is contained in $N$.
\item[] Assume that $H$ is a subgroup of $G$ not contained in $N$. Then $H$ has an element of the form $z= kx+t$, for some non-zero $k\in
\N$ and $t\in N$. Since $t$ is torsion, there exists $m>0$ such that $mz=mkx\in H$. Since $K$ is divisible, for every $j\in\N$ we have
$$
K=jK=j\ol{\langle x\rangle}\subseteq\ol{\langle jx\rangle}.
$$
In particular for $j=mk$ we obtain $K\subseteq\ol{\langle mkx\rangle}\subseteq H$. This proves that $H$ is dense in $G$ whenever $H$ is not contained in $N$.

\item[(xii)]  Every minimal subgroup of $G$ is contained in $N$.
\item[] Let us show first that $N$ is minimal. According to  the minimality criterion (Theorem \ref{Min_Crit}), we have to show that $N$ is essential
in the compact group $L_0$ and hence in $\ol{N}\le L_0$. This follows from (viii).

Since $G$ itself is not minimal (as it has trivial intersection with the closed subgroup $\hull{h(a_2)}$ of $K$), no dense subgroup of $G$ can be
minimal by the minimality criterion. Therefore, the closure of every minimal subgroup of $G$ is a proper subgroup of $G$ and hence contained in $N$ by (xi).
\item[(xiii)] $G/N$ is not UFSS.
\item[] According to Proposition \ref{perm_prop_3} (a) and the Grant-Sulley lemma (see the proof in \cite[Lemma 4.3.2]{DPS}) this is equivalent to: $K/\ol{N}$ is not UFSS.
By Proposition \ref{perm_prop_3}(b) it is sufficient to show that its subgroup $L/\ol{N}$ is not UFSS. By (v), $L/\ol{N}\cong\Z(4)\times \ol{N}$ is an infinite profinite group and hence not UFSS.
\end{itemize}
\end{example}
\begin{remark} \label{crit_almost}
This example shows that the counterpart for almost minimality of the criterion in Theorem \ref{crit}
fails: a locally essential dense subgroup of a compact group may fail to be almost minimal.
\end{remark}
The above construction of the group $G$ depends on the choice of the topological generator $x$ of $K$. This is why we denote it in the sequel by $G_x$. It is known that  the  set $gen(K)$ of topological generators of $K$ is a dense $G_\delta$-set of $K$ of Haar measure 1 \cite{HS}. In particular the following Proposition implies that one has ${\mathfrak c}$-many pairwise non-isomorphic groups $G$ with the properties of the above example.

\begin{proposition}\label{Rem_Anti_StrongLM_}
With the above notations, for $x,\,z\in gen(G)$
$$
G_x\cong G_z\mbox{ as topological groups } \Longleftrightarrow \;G_x= G_z  \;\Longleftrightarrow\;  z=\pm x. \eqno(*)
$$
\end{proposition}
\begin{proof} Both implications ``$\Leftarrow$" are obvious. Assume that $f: G_x\to G_z$ is an isomorphism. Then its extension to the common completion $K$ provides
a topological  isomorphism $g: K \to K$. The automorphism $g$ of $K$ can be lifted to a continuous automorphism  $\eta: \Q^{\wedge} \to \Q^{\wedge}$, i.e., $g\circ h = h \circ \eta$.
To prove this just note that the injectivity of the homomorphisms ${h}^{\wedge}: K^{\wedge}\to \Q^{\wedge \wedge}\cong \Q$ allows us to consider $K^{\wedge}$ as a subgroup of $\Q$,
so the endomorphism ${g^{\wedge}}: K^{\wedge} \to K^{\wedge}$ can be extended to an endomorphism $\xi: {\Q} \to {\Q}$ of the divisible group $\Q$ such that $\xi \circ {h}^{\wedge}=
{h}^{\wedge} \circ {g}^{\wedge}$. Being $\xi:\Q \to \Q$ nontrivial, it is an automorphism. Now choose $\eta$ with $\xi={\eta}^{\wedge}$.
$$\xymatrix{{\mathbb Q}  \ar[r]^{\xi} &{\mathbb Q} \\  {K^\wedge}  \ar[r]^{g^\wedge}\ar[u]^{h^\wedge} &{K^\wedge}\ar[u]^{h^\wedge}
} $$

Since $g$ is an automorphism  of $K$, the equality $g\circ h = h \circ \eta$ yields
$$
\eta(\ker h) = \ker h=\Z^\perp.
$$
%

%

%

Observe that ${\mathbb Q}^{\wedge}$ is a divisible torsion-free abelian group, so that multiplication by rationals $r\in {\mathbb Q} $
makes sense. There  exists  $0\ne r\in \Q$ such that $\xi(1)=r$ and it follows easily that $\eta(t)=r t$ for all $t\in \Q^{\wedge}$. Say $r=a/b$, with $a, b \in \Z\setminus \{0\}$. To prove that $g = \pm id_K$ it suffices to prove that $\eta= \pm id_{ \Q^{\wedge}}$.
Since for every $p$ the group ${\mathbb H}_p$ is $q$-divisible for every prime $q\ne p$, one has
$\eta({\mathbb H}_p)= p^{v_p(r)}{\mathbb H}_p$, where $v_p(r)$ is determined by $r=p^{v_p(r)}\frac{a'}{b'}$, where the integers $a', b'$ are coprime to $p$ and $(a',b')=1$. Thus $\eta (\bigoplus_p {\mathbb H}_p)=\bigoplus_p p^{v_p(r)}{\mathbb H}_p$. Note that only finitely many $v_p(r)$ may be distinct from 0 (namely, for the prime divisors $p$ of $a$ or $b$), so $ \overline{\bigoplus_p p^{v_p(r)}{\mathbb H}_p}= \prod _p p^{v_p(r)}{\mathbb H}_p$.   Therefore, by the density of $\bigoplus_p p^2 {\mathbb H}_p$ in $\ker h=\prod _p p^{2}{\mathbb H}_p$
$$
\eta (\ker h)= \eta \left(\overline{\bigoplus_p p^2 {\mathbb H}_p}\right)=\overline{\eta\left( \bigoplus_p p^2 {\mathbb H}_p\right)}=
\overline{\bigoplus_p p^2\eta( {\mathbb H}_p)}=\prod _p p^{v_p(r)+2}{\mathbb H}_p.
$$
This yields $\prod _p p^{v_p(r)+2}{\mathbb H}_p= \prod _p p^{2}{\mathbb H}_p$,
therefore $v_p(r)=0$ for all $p$. Hence $r=\pm 1$. This proves that $\eta=\pm id_K$. In particular,
$G_x= G_z$.

To prove the second implication in $(*)$ assume that $G_z= G_x$.
From $x\in G_z$ and $z\in G_x$ we conclude that $x=kz+t_1$ and $z=mx +t_2$.
Therefore, $x=kmx+t_1+ kt_2$. Hence $(1-km)x\in N$ is torsion. As $x$ is non-torsion, this is possible only when $mk=1$, i.e., $z=\pm x$.
\end{proof}

\begin{remark}\label{Rem_Anti_StrongLM}
\begin{itemize}
   \item[(a)] It can be proved that in a certain sense the above example is the smallest possible example of a precompact locally minimal group that fails to be almost minimal.
In fact, we show elsewhere that every precompact locally minimal group topology on the free abelian group $\Z^n$ is almost minimal.
    \item[(b)] One can produce in a similar way a precompact locally minimal non-almost minimal group topology also on the free group $\bigoplus_\omega \Z$. Indeed, take ${\mathbb H}_p$, $a_p$ as above, fix a topological generator $y$ of $\Q^{\wedge}$ and a prime $q$, and put now $K=\Q^{\wedge}/{\mathbb H}_q$, with $h: \Q^{\wedge}\to K$ the canonical map. For $p\ne q$ let $F_p= \hull{h(a_p)}$, $N= \bigoplus_{p\ne q}\ pF_p$, $x=h(y)$ and $G =\hull{x}\oplus N$. Then $G$ is a free abelian group of rank $\omega$, it is a dense locally minimal
subgroup of $K$, $N$ is the largest minimal closed subgroup of $G$ and $G/N$ is not UFSS. Consequently, $G$ is not almost minimal.
      \end{itemize}
\end{remark}

\section{Open questions}

Here we collect some open questions concerning almost minimality.

\begin{question}\label{Ques1} Let $U\subseteq G$ witness almost minimality and let
$V\subseteq U$ be a GTG neighborhood of $0$. Does $V$ also witness almost minimality?
\end{question}

Minimal abelian groups are precompact, hence locally GTG. According
to Theorem \ref{char_alm_min}
every almost minimal abelian group $G$ has a GTG neighborhood  $U$ of zero  such that  $G$
is  $U$-locally minimal. Moreover, every maximally almost periodic, almost minimal group is locally GTG (Proposition \ref{map_gtg}). So it is natural to ask whether the hypothesis ``maximally almost periodic"{} can be dropped:

\begin{question}
Is every almost minimal abelian group a locally GTG group?
\end{question}
Taking closed subgroups preserves both UFSS and minimality. This suggests the following

\begin{question}\label{Ques2} Is almost minimality preserved by taking closed subgroups?
\end{question}

\begin{question}\label{Ques5} Under which conditions is a dense subgroup $H$ of an almost minimal group $G$, almost minimal? A necessary condition is ``$H$ is locally essential"{} but according to Example \ref{counterexample}, it is not sufficient (see Remark \ref{crit_almost}).
\end{question}

\begin{question}\label{Ques3} (compare with Corollary \ref{completion_slm}) Are the complete locally GTG locally minimal groups also almost minimal?
\end{question}
We showed in \cite[Corollary 5.9]{LocMin} that the completion of a locally GTG group is locally GTG.
Hence, Question \ref{Ques3} is equivalent to the following
\begin{question}\label{Ques4} (compare with Corollary \ref{completion_slm2}) Is the completion of a locally GTG locally minimal group always almost minimal?
\end{question}

%
%
%
%



\end{document}